\documentclass[10pt,a4paper,twoside]{amsart}
\usepackage{graphicx}
\usepackage{a4wide}
\usepackage{color}
\usepackage{units}
\usepackage[colorlinks]{hyperref}
\usepackage{upref}
\usepackage{xifthen}
\usepackage{mathtools}
\graphicspath{{figures/}}

\newtheoremstyle{note}% name
  {3pt}%      Space above
  {3pt}%      Space below
  {}%         Body font
  {}%         Indent amount (empty = no indent, \parindent = para indent)
  {\itshape}% Thm head font
  {:}%        Punctuation after thm head
  {.5em}%     Space after thm head: " " = normal interword space;
  {}%         Thm head spec (can be left empty, meaning `normal')

\newtheorem{lemma}{Lemma}[section]

\newtheorem{corollary}[lemma]{Corollary}
\newtheorem{theorem}[lemma]{Theorem}
\newtheorem{remark}[lemma]{Remark}

\newtheorem*{maintheorem*}{Main Theorem}
\theoremstyle{definition}{\newtheorem{definition}[lemma]{Definition}}

\theoremstyle{note}{\newtheorem{claim}{Claim}[section]
\newtheorem*{claim*}{Claim}}

\allowdisplaybreaks
\numberwithin{equation}{section}
\newcommand{\abs}[1]{\left|#1\right|}
\newcommand{\norm}[1]{\left\| #1 \right\|}

\newcommand{\wde}{w^\delta}
\newcommand{\xdom}{D}
\newcommand{\Nx}{N_D}
\newcommand{\ad}{a^\delta}
\newcommand{\oa}{\overline{a}}
\newcommand{\ua}{\underline{a}}

\newcommand{\Dom}{D_T}

\renewcommand{\i}{\ifmmode\mathit{\mathchar"7010 }\else\char"10 \fi}
\renewcommand{\j}{\ifmmode\mathit{\mathchar"7011 }\else\char"11 \fi}
\newcommand{\R}{\mathbb{R}}
\newcommand{\N}{\mathbb{N}}
\newcommand{\Z}{\mathbb{Z}}
\newcommand{\seq}[1]{\left\{#1\right\}}

\newcommand{\bx}{\mathbf{x}}

\newcommand{\wh}{\widehat{w}}
\newcommand{\epsx}{\epsilon_1}
\newcommand{\epst}{\epsilon_0}

\newcommand{\Dt}{{\Delta t}}
\newcommand{\Dx}{{\Delta x}}
\newcommand{\Dm}{D^-_x}
\newcommand{\Dp}{D_x^+}
\newcommand{\Dc}{D^c_x}
\newcommand{\Dym}{D^-_y}
\newcommand{\Dyp}{D_y^+}
\newcommand{\Dyc}{D^c_y}
\newcommand{\Dtp}{D_t^+}

\newcommand{\Epsilon}{\mathcal{E}}
\newcommand{\duz}{d\underline{z}}
\newcommand{\Dxm}{\Dm}
\newcommand{\Dxp}{\Dp}
\newcommand{\Dxc}{\Dc}
\renewcommand{\Dxp}[1][]{%
   \ifthenelse{\isempty{#1}{}}{D^+_x}{D^+_{x_{#1}}}%
 }
\renewcommand{\Dxm}[1][]{%
   \ifthenelse{\isempty{#1}{}}{D^-_x}{D^-_{x_{#1}}}%
 }
\renewcommand{\Dyp}[1][]{%
   \ifthenelse{\isempty{#1}{}}{D^+_y}{D^+_{y_{#1}}}%
 }
\renewcommand{\Dym}[1][]{%
   \ifthenelse{\isempty{#1}{}}{D^-_y}{D^-_{y_{#1}}}%
 }
\renewcommand{\Dxc}[1][]{%
   \ifthenelse{\isempty{#1}{}}{D^c_x}{D^c_{x_{#1}}}%
 }

\newcommand{\wdi}{w_{\Dx}}

{%

\begin{enumerate}}%
{\end{enumerate}}

{%

\begin{enumerate}}%
{\end{enumerate}}

\begin{document}%\large

\title[Rough coefficients] {Convergence rates of finite difference schemes for the linear advection and wave equation with rough coefficient}

\date{\today}
\author[]{}
 \author[F. Weber]{Franziska Weber} \address[F. Weber]{\newline
   Seminar for Applied Mathematics (SAM) \newline ETH Z\"urich,
   \newline R\"amistrasse 101, Z\"urich, Switzerland.}
 \email[]{franziska.weber@sam.math.ethz.ch}
 \urladdr{http://www.sam.math.ethz.ch/\~{}frweber}



\begin{abstract}
We prove convergence rates of explicit finite difference schemes for the linear advection and wave equation in one space dimension with H\"{o}lder continuous coefficient. The obtained convergence rates explicitely depend on the H\"{o}lder regularity of the coefficient and the modulus of continuity of the initial data. We compare the theoretically established rates with the experimental rates of a couple of numerical examples. 
 \end{abstract}

\maketitle

\section{Introduction}

Propagation of acoustic waves in a heterogeneous medium plays an important
role in many applications, for instance in seismic imaging
in geophysics and in the exploration of hydrocarbons
\cite{BIO1,IA1}. This wave propagation is modeled by the linear wave
equation:
\begin{subequations}\label{eq:waveeq}
  \begin{align}
    \label{seq1:wavepde}
    \frac{1}{c^2(\mathbf{x})} \partial_{tt}^2 p(t,\mathbf{x})-\Delta p(t,\mathbf{x})&=0, \quad (t,\mathbf{x})\in D_T,\\
    \label{seq2:initwave1}
    p(0,\bx)&=p_0(\bx), \quad \bx\in D,\\
    \label{seq3:inti2}
    \partial_t p(0,\bx)&=p_1(\bx),\quad \bx\in D,
  \end{align}
\end{subequations}
where $D_T:=[0,T]\times D$, $D\subset \R^d$, augmented with boundary conditions. 
Here, $p$ is the acoustic
pressure and the wave speed is determined by the coefficient $c^2 =
c^2(\mathbf{x}) > 0$. The coefficient $c$ encodes information about the
material properties of the medium. As an example, the coefficient $c$
represents various geological properties  when seismic waves propagate in a
rock formation.
It is well known that the linear wave equation \eqref{eq:waveeq} can
be rewritten as a first-order system of partial differential equations
by defining $v(t,\bx):=\partial_t p(t,\bx)$ and $\mathbf{u}(t,\bx):=\nabla p(t,\bx)$,
resulting in
\begin{subequations}
  \label{eq:wavesys2}
  \begin{align}
    \label{seq1:system12}
    \frac{1}{c^2(\bx)}\partial_t v(t,\bx)-\mathrm{div}( \mathbf{u}(t,\bx))&=0,\\
    \partial_t\mathbf{u}(t,\bx)-\nabla v(t,\bx)&=0, \quad (t,\mathbf{x})\in  D_T,\nonumber\\
    \label{seq2:initwave122}
    v(0,\bx)&=p_1(\bx), \quad \bx\in D,\\
    \label{seq3:inti22}
    \mathbf{u}(0,\bx)&=\nabla p_0(\bx),\quad \bx\in D.
  \end{align}
\end{subequations}

The above system \eqref{eq:wavesys2} is strictly hyperbolic
\cite{GKO1} with wave speeds given by $\pm {c}$. Under the
assumption that the coefficient $c^2\in C^{0,\alpha}\cap L^\infty(D)$
for some $\alpha>0$ and that it is uniformly positive on $D$, i.e., there
exists constants $\underline{c}, \overline{c} > 0$ such that
\begin{equation}
  \label{eq:assona}
  0 < \underline{c} \leq c^2({\mathbf x}) \leq \overline{c}, \quad \forall {\mathbf x} \in D.
\end{equation}
and that the initial data $p_0\in H^{1}(D)$ and $p_1\in L^2(D)$, one
can prove existence of a unique weak solution $p\in C^0([0,T];H^1(D))$
with $\partial_t p \in C^0([0,T];L^2(D))$ following classical energy arguments
for linear partial differential equations. See for instance
\cite[Chapter III, Theorems 8.1 and 8.2]{LM1972}. A smoother
coefficient $c$ and more regular initial data $p_0,p_1$ result in a
more regular solution \cite{LM1972}.

Even though equations \eqref{eq:waveeq}, \eqref{eq:wavesys2} are linear, analytical solution formulae are in general not available due to the possibly complex geometry of the domain $D$, the heterogeneity of the coefficient, or boundary conditions. Consequently, solutions have to be approximated numerically. Among the most popular methods for the linear wave equation with inhomogeneous coefficient $c$ are finite difference methods and finite element methods for which an extensive stability and convergence analysis is available~\cite{GKO1,KL1,LT1}. A key question is the rate at which numerical schemes converge to the exact solution as the discretization paramter goes to zero or the computational effort increases, since this allows to estimate the computational work needed to get a certain desired quality of the approximation. If the underlying solution of the equation is smooth, this generally depends on the order of the truncation error which is determined by the order of the spatial and temporal discretization, that is, the order of the underlying difference operators (for finite difference schemes) or the dimension of the polynomial approximation spaces (for finite element methods). The smoothness of the solution again depends on the regularity of the coefficient and the initial data for the equation.
If the coefficient $c$ and the initial data $p_0,p_1$
are smooth, say $C^k(D)$ or $H^s(D)$ for some large enough Sobolev
exponent $s$, then by regularity results for the linear wave equation
\cite{LM1972}, the solution is also smooth, i.e., it belongs to
$H^s(D_T)$ and the finite difference (resp. finite element)
discretizations converge at the order of the underlying difference
operators (resp. polynomial approximation spaces).
 
\subsection{Rough coefficients}
As noted above, the regularity of the solution to the wave equation
\eqref{eq:waveeq} and the resulting (high) rate of convergence of
numerical approximations rely on the smoothness of the coefficient
$c$. Accordingly, most of the numerical analysis literature on the
wave equation assumes a smooth coefficient $c$. However, this
assumption is not always realized in practice. As noted before, the wave
equation is heavily used to model seismic imaging in rock formations
and other porous media (for instance oil and gas reservoirs). Such
media are very heterogeneous with sharp interfaces, strong contrasts
and aspect ratios \cite{IA1}. Furthermore, the material properties of
such media can only be determined by measurements. Such measurements
are inherently \emph{uncertain}. This uncertainty is modeled in a
statistical manner by representing the material properties (such as
rock permeability) as random fields. In particular, log-normal random
fields are heavily used in modeling material properties in porous and
other geophysically relevant media \cite{IA1,FGPS1}. Thus, the
coefficient $c$ is not smooth, not even continuously differentiable,
see Figure \ref{fig:coef1d} for an illustration of coefficient $c$
where the rock permeability is modeled by a log-normal random field
(the figure represents a single realization of the field). Closer
inspection of the coefficients obtained in practice reveals that the
material coefficient $c$ is at most a H\"older continuous function,
that is, $c \in C^{0,\alpha}$ for some $0 < \alpha < 1$. No further
regularity can be assumed on the coefficient $c$ representing material
properties of most geophysical formations.

Given the above discussion, it is natural to search for numerical
methods that can effectively and efficiently approximate the acoustic
wave equation with rough (merely H\"older continuous) coefficients. In
particular, one is interested in designing numerical methods that can
be rigorously shown to converge to the underlying weak solution. Furthermore, one is also interested in obtaining a convergence rate for the discretization as
the mesh parameters are refined. We remark that the issue of a
convergence rate is not just of theoretical significance, it has for example
profound implications on calculating complexity estimates for Monte
Carlo and Multilevel Monte Carlo methods (see
\cite{sssid3}) to solve the random (uncertain) PDE that
results from considering the material coefficient as a random field
(as is done in engineering practice).

A search through the literature revealed that there are not many results available concerning convergence rates for linear hyperbolic partial differential equations with rough coefficients. In~\cite{Jovanovic1987}, B. Jovanovi\'{c}, L. Ivanovi\'{c} and E. S\"{u}li prove convergence rates for finite difference approximations of the hyperbolic system
\begin{equation*}
\frac{\partial^2 u}{\partial t^2} -\sum_{i,j=1}^2 \frac{\partial}{\partial x_i}\left(a_{ij}\frac{\partial u}{\partial x_j}\right)+a u = f,\quad \text{in } (0,T)\times D,
\end{equation*}
under the assumption that the solution $u$ lies in the Sobolev space $H^\lambda((0,T)\times D)$ and the coefficients $a_{ij}(x)\in W^{\lambda-1,\infty}(D)$, $a(x)\in W^{\lambda-2,\infty}(D)$, with $2<\lambda\leq 4$. It is shown that the approximations converge at rate $\Dx^{\lambda-2}$ in the energy norm (a discrete version of the $H^1$-norm). In~\cite{Jovanovic1992}, this result is extended to coefficients $a_{ij}(x)\in W^{\lambda-1,2}(D)$, $a(x)\in W^{\lambda-2,2}(D)$.

Another related work is the article~\cite{Jovanovic2005} by V. Jovanovi\'{c} and C. Rohde where they establish error estimates for finite volume approximations of linear hyperbolic systems in multiple space dimensions for initial data with low regularity, however under the assumption that the coefficients are smooth.

A survey of the available results on finite difference methods for hyperbolic equations can be found in Chapter 4 of the book by B. Jovanovi\'{c} and E. S\"{u}li~\cite{FDMbook}.

Given this paucity of available results, we aim to contribute to the theory with the current paper. Since our ultimate goal is proving convergence rates for multidimensional wave equations, we will start by analyzing rates for linear hyperbolic systems in one space dimension, in particular, we will consider the linear advection equation
\begin{equation}\label{i:linadv}
\partial_t u(t,x)+\partial_x (a(x)u(t,x))=0, \quad (t,x)\in [0,T]\times D,
\end{equation}
and the linear wave equation
\begin{equation}\label{i:1dwave}
\frac{1}{a(x)}\partial^2_{tt} p(t,x) -\partial_{xx}^2 p(t,x)=0,\quad (t,x)\in [0,T]\times D,
\end{equation}
in one space dimension where the spatially heterogeneous coefficient $a$ is positive and uniformly bounded but has only little regularity, specifically, we assume $a\in C^{0,\alpha}(D)$, for some $0<\alpha\leq 1$. We require the initial data $u_0$ for the transport equation to be H\"{o}lder continuous with exponent $\gamma>0$, and for the wave equation, we assume that the derivatives $\partial_t p$ and 
$\partial_x p$ have moduli of continuity in $L^2$.

These equations can be seen as prototype models for equation \eqref{eq:waveeq} above for $a:=c^2$. 

Equation \eqref{i:linadv} also appears as a model for transport of pollutants in heterogeneous media, see for example~\cite{Elfeki2012,GWAT199}. Moreover, these two models are related to mixing in turbulent flow~\cite{turbulent}. 

In a first part (Section \ref{sec:intro}), we study the properties of equation \eqref{i:linadv} under the assumption that $a$ is H\"{o}lder continuous and propose a simple upwind scheme for the numerical approximation. We show that this scheme is stable under a linear CFL-condition and converges, and then prove a convergence rate of the scheme in $L^1(\xdom)$ and $L^2(\xdom)$ depending explicitly on the H\"{o}lder regularity of the coefficient $a$ and the initial data. To prove the rate, we show that the numerical approximations are approximately H\"{o}lder continuous in time and use a variant of S.~N.~Kru\v{z}kov's doubling of variables technique~\cite{k1970} combined with a type of Gr\"{o}nwall inequality for the $L^2$-case. We conclude the section with a couple of numerical experiments that confirm that the rates are indeed quite low, but higher than the theoretically established rate. This may indicate that our estimate is not sharp.

In the second part (Section \ref{sec:1d}), we show that the techniques from Section \ref{sec:intro} can be used to establish a convergence rate for an upwind finite difference scheme for the first order reformulation \eqref{eq:wavesys2} of the linear wave equation \eqref{i:1dwave} in one space dimension for H\"{o}lder continuous coefficient $a$ but under slightly stronger assumptions on the initial data, or given that the solution has a known modulus of continuity. Again, we conduct a couple of numerical experiments and observe that the experimental rates are close to the theoretically derived ones. We conclude by summarizing the results and suggesting further research directions in Section \ref{sec:conc}.

\section{Transport equation with H\"{o}lder continuous coefficient}
\label{sec:intro}

The purpose of this section is to investigate properties the linear advection equation in one space dimension,
\begin{subequations}\label{eq:setup}
  \begin{align}\label{seq1:trans1}
    \frac{\partial}{\partial t}u(t,x)+\frac{\partial}{\partial
      x}(a(x)u(t,x))=0,&\quad (t,x)\in \Dom,\\ 
    \label{seq2:inittrans}
    u(0,x)=u_0(x),&\quad x\in \xdom,
  \end{align}
\end{subequations}
on the domain $\Dom:=(0,T]\times \xdom$, for some finite time $T>0$, a finite interval $0\in\xdom\subset \R$ with length $|D|\geq 2$, periodic boundary conditions and $u_0\in L^1(\xdom)$ a given initial data. Alternatively, we could consider $D=\R$ and compactly supported initial data $u_0\in L^1(\R)$.We consider coefficients
$a\in L^{\infty}(\xdom)$ which are positive and bounded away from zero,
that is
\begin{equation}\label{eq:ascona}
  \oa \geq a(x)\geq \ua>0,\quad \forall\, x\in\xdom.
\end{equation}
as well as H\"{o}lder continuous, $a\in C^{0,\alpha}(\overline{\xdom})$, for some exponent $\alpha>0$. 
As we will see in the following, it is more convenient to work with the variable $w:=a(x) u$ instead of $u$ and the equation it satisfies,
\begin{subequations}\label{eq:forw}
  \begin{align}\label{seq1:weq}
    \frac{\partial}{\partial t}\biggl(\frac{w(t,x)}{a(x)}\biggr)+ \frac{\partial}{\partial x}w(t,x)=0,&\quad (t,x)\in \Dom,\\
    \label{seq2:initw}
    w(0,x)=w_0(x):=a(x)u_0(x),&\quad x\in \xdom.
  \end{align}
\end{subequations}
We assume that the initial data $w_0$ is H\"{o}lder continuous $C^{0,\gamma_\infty}(\xdom)$ for some $\gamma_\infty>0$:
\begin{equation}\label{eq:mocminfty}
|w_0|_{C^{0,\gamma_\infty}(\xdom)}:=\sup_{x\neq y\in\xdom}\frac{|w_0(x)-w_0(y)|}{|x-y|^{\gamma_\infty}}\leq C <\infty.
\end{equation}
We note that this implies in particular for any $x\in\xdom$
\begin{equation*}
 \sup_{|h|\leq \sigma}|w_0(x+h)-w_0(x)|\leq C \, \sigma^{\gamma_\infty}.
\end{equation*}
We will see that some solution properties can also be obtained under the slightly weaker assumption
\begin{equation}\label{eq:mocinit}
\int_D \sup_{|h|\leq \sigma}|w_0(x+h)-w_0(x)|^p\, dx\leq \sigma^{p\gamma_p}
\end{equation}
In the following, we investigate to what extent the exponents $\alpha$, $\gamma_{p}$ influence the regularity of the solution $w$ at a time $t>0$ and the convergence rate of the finite difference scheme
\begin{equation}\label{eq:upwindscheme}
  \frac{D^+_t w^{n}_j}{a_j}=-D^-_x w_j^n, \quad 1\leq j\leq \Nx,\  0\le n\le N_T.
\end{equation}
as the mesh is refined.
\subsection{Regularization of the coefficient}
Since the coefficient $a$ is not differentiable, it is possible that the solution $w$ is not differentiable in the classical sense either and only weak solutions to equation \eqref{seq1:weq} can be defined. By a weak solution to \eqref{eq:forw}, we mean a function $w=w(t,x)\in C^{0,\gamma}(\Dom)$ for some $\gamma>0$ satisfying \eqref{eq:forw} in the distributional sense, that is, for all smooth test functions $\varphi\in C^{\infty}_c(\Dom)$,
\begin{equation*}
\int_{\Dom} \frac{w}{a}\,\partial_t\varphi\, dxdt + \int_{\Dom} w \,\partial_x\varphi\,dxdt + \int_D \frac{w_0(x)}{a}\varphi(0,x)\, dx = \int_D \frac{w(t,x)}{a}\varphi(t,x)\, dxdt.
\end{equation*}
To deal with the possible non-differentiability of the solution, we will in a first step regularize the coefficient $a$ by convolving it with a smooth test function $\omega_\delta\in C^{\infty}_0(\R)$, given as
\begin{equation}\label{eq:omegadelta}
  \omega_{\delta}(x)=\frac{1}{\delta} \omega\biggl(\frac{x}{\delta}\biggr),
\end{equation}
where $\omega\in C^{\infty}_0(\R)$ is an even function with the properties
\begin{equation*}
  0\le \omega\le 1,\quad \omega(x)=0 \textrm{ for } |x|\ge 1,\quad \int_{\R} \omega(x)\,dx=1.
\end{equation*}
Then we consider the solution $\wde$ of the equation
\begin{subequations}\label{eq:forwdelta}
  \begin{align}\label{seq1:weqd}
    \frac{\partial}{\partial t}\biggl(\frac{\wde(t,x)}{\ad(x)}\biggr)+ \frac{\partial}{\partial x}\wde(t,x)=0,&\quad (t,x)\in \Dom,\\
 \label{seq2:initwd}
    \wde(0,x)=\wde_0(x):=(w\ast \omega_\delta)(x),&\quad x\in \xdom,\\
\label{seq3:ad}
\ad(x):=(a\ast \omega_\delta)(x),&\quad x\in \xdom.
  \end{align}
\end{subequations}
The coefficient $\ad$ is smooth and therefore in particular Lipschitz continuous and hence we can define classical solutions to \eqref{eq:forwdelta} in a standard way, using the method of characteristics: We let $\eta$ solve the ordinary differential equation
\begin{equation}\label{eq:char}
\begin{split}
\frac{d}{dt}\eta(t,x_0)=\ad(\eta(t,x_0)),&\quad (t,x_0)\in \Dom\\
\eta(x_0,0)=x_0,&\quad x_0\in\xdom.
\end{split}
\end{equation}
$\wde$ is constant along the charachteristics $\eta$ since
\begin{align*}
\frac{d}{dt}\wde(t,\eta(t,x_0))&=\partial_t \wde(t,\eta(t,x_0)+\frac{d}{dt}\eta(t,x_0) \partial_x\wde(t,\eta(t,x_0))\\
&=\ad(\eta(t,x_0))\left( \frac{\partial_t\wde(t,\eta(t,x_0))}{\ad(\eta(t,x_0))}+ \partial_x \wde(t,\eta(t,x_0))\right)=0.
\end{align*}
Thus the solution at time $t>0$ is given by
\begin{equation*}
\wde(t,\eta(t,x_0))=\wde_0(x_0).
\end{equation*}
Therefore, if the initial data is in $L^\infty(\overline{\xdom})$, the solution will be essentially bounded at any later time. Using this, we can derive H\"older continuity of the solution in time and space:
\begin{lemma}\label{lem:hoeldercont}
Assume that the coefficient $a$ in \eqref{eq:forwdelta} is  bounded, i.e. satisfying \eqref{eq:ascona}, and that $w_0$ is H\"older continuous with exponent $\gamma_\infty$, as in \eqref{eq:mocminfty}. Then $\wde$ will be H\"older continuous in space and time with exponent $\gamma_\infty$ for any time $t>0$, independently of $\delta>0$. In particular, we have
\begin{align*}
\sup_{x,t\in \Dom,|h|\leq\sigma}\frac{|\wde(t,x)-\wde(t+h,x)|}{h^{\gamma_{\infty}}}&\leq \oa^{\gamma_\infty} \|w_0\|_{C^{0,\gamma_{\infty}}},\\
\sup_{x,t\in \Dom,|h|\leq\sigma}\frac{|\wde(t,x+h)-\wde(t,x)|}{h^{\gamma_{\infty}}}&\leq \left(\frac{\oa}{\ua}\right)^{\gamma_\infty} \|w_0\|_{C^{0,\gamma_{\infty}}},\\
\end{align*}
\end{lemma}
\begin{proof}
We consider the characteristics equation \eqref{eq:char} once more and note that it is independent of the initial data $w_0$ of $w$. Hence any initial data will be propagated along the same characteristics, and we have for the difference $\wh:=\wde_1-\wde_2$ corresponding to initial condition $\wh_0:=\wde_{0,1}-\wde_{0,2}$
\begin{equation*}
\wh(t,\eta(t,x_0))=\wh_0(x_0)
\end{equation*}
with $\eta$ defined in \eqref{eq:char}. In particular, taking $\wde_{2,0}:=\wde_1(h,\cdot)$ for some $h>0$ (the case $h<0$ is analoguous),  and omitting the index $1$ this implies
\begin{equation*}
\|\wde(t,\cdot)-\wde(t+h,\cdot)\|_{L^{\infty}}\leq \|\wde_0-\wde(h,\cdot)\|_{L^{\infty}}.
\end{equation*}
The characteristics equation, \eqref{eq:char}, implies for any $x\in \xdom$,
\begin{equation*}
x=\eta(h,x_0^x)=x_0^x+\int_0^h \ad(\eta(s,x_0^x))\, ds
\end{equation*}
for some $x_0^x$. By the assumption on the boundedness of $\ad$, \eqref{eq:ascona}, we can bound $x_0^x$ from above and below:
\begin{equation}\label{eq:egal}
x-h \oa\leq x-\int_0^h \ad(\eta(s,x^x_0))\, ds= x_0^x\leq x
\end{equation}
and thus for any $x\in\xdom$
\begin{align*}
 |\wde_0(x)-\wde(h,x)| &=  |\wde_0(x)-\wde_0(x_0^x))|\\
&\leq  \sup_{y\in[x-\oa h,x]} |\wde_0(x)-\wde_0(y))|\\
&\leq |\wde_0|_{C^{0,\gamma_\infty}}  \oa^{\gamma_\infty}h^{\gamma_\infty}\\
&\leq |w_0|_{C^{0,\gamma_\infty}}  \oa^{\gamma_\infty}h^{\gamma_\infty}
\end{align*}
by the assumption on the initial data. Taking the supremum over all $x\in\xdom$, we obtain
\begin{equation*}
\|\wde(t,\cdot)-\wde(t+h,\cdot)\|_{L^{\infty}}\leq \oa^{\gamma_\infty} \|w_0\|_{C^{0,\gamma_{\infty}}} h^{\gamma_{\infty}}
\end{equation*}
and thus the H\"older continuity in time. To prove the H\"older continuity in space, we note that, by the characteristics equation \eqref{eq:char} and the positivity of the coefficient $\ad$, we have
\begin{equation*}
\wde(x+h,t)=\wde(x,\tau^x),
\end{equation*}
for some $\tau^x<t$ such that 
\begin{equation}\label{eq:gugus}
x+h=x+\int_{\tau^x}^t \ad(\eta(s,x))\, ds
\end{equation}
(the characteristics starting at $(\tau^x,x)$ and passing through $(t,x+h)$). This allows us to bound $\tau^x$ from below:
\begin{equation*}
h=\int_{\tau^x}^t \ad(\eta(s,x))\, ds\geq (t-\tau^x) \ua
\end{equation*}
and therefore
\begin{equation}\label{eq:gaga}
\tau^x\geq t - h \ua^{-1}.
\end{equation}
This estimate is independent of $x$ and $\delta$.
Hence
\begin{align*}
|\wde(x,t)- \wde(x+h,t)|&=|\wde(x,t)- \wde(x,\tau^x)|\\
&\leq  \sup_{\tau\in [t-h\ua^{-1},t)}|\wde(x,t)- \wde(x,\tau)|\\
&\leq \|w_0\|_{C^{0,\gamma_{\infty}}} \left(\frac{\oa}{\ua}\right)^{\gamma_\infty} h^{\gamma_\infty},
\end{align*}
which yields the H\"older continuity in space.
\end{proof}
\begin{remark}\label{rem:moclp}
We note that the estimates in Lemma \ref{lem:hoeldercont} are independent of $\delta>0$ and the H\"older coefficient $\alpha$ of $a$. Moreover, if $\xdom$ is bounded, or the initial data $w_0$ has compact support, this Lemma implies that $\wde$ has a modulus of continuity in time, 
%\begin{subequations}\label{eq:mocwd}
\begin{equation}%\label{seq1:mocwd1}
%\nu_t^1(\wde(t,\cdot),\sigma)&:=\sup_{|h|\leq\sigma}\int_{\xdom} |\wde(t+h,x)-\wde(t,x)|\, dx\leq C \sigma^{\gamma_1},\\
\label{seq2:mocwd2}
\nu_t^p(\wde(t,\cdot),\sigma):=\sup_{|h|\leq\sigma}\int_{\xdom} |\wde(t+h,x)-\wde(t,x)|^p\, dx\leq C \sigma^{p\gamma_p},
\end{equation}
%\end{subequations}
and a modulus of continuity of the same order in space:
%\begin{subequations}\label{eq:mocwdx}
\begin{equation}%\label{seq1:mocwd1x}
%\nu_x^1(\wde(t,\cdot),\sigma)&:=\sup_{|h|\leq\sigma}\int_{\xdom} |%\wde(t,x+h)-\wde(t,x)|\, dx\leq C \sigma^{\gamma_1},\\
\label{seq2:mocwd2x}
\nu_x^p(\wde(t,\cdot),\sigma):=\sup_{|h|\leq\sigma}\int_{\xdom} |\wde(t,x+h)-\wde(t,x)|^p\, dx\leq C \sigma^{p\gamma_p}.
\end{equation}
%\end{subequations}
with $\gamma_p=\gamma_\infty$ for all $p\in[1,\infty)$.
\end{remark}
\begin{remark}[Entropy identity]
Subtracting a constant $k$ from equation \eqref{seq1:weqd} and then multiplying by a regularized version of the sign function, we obtain the $L^1$-norm conservation property of the equation:
\begin{equation}\label{eq:l1cons}
\int_{\xdom} \frac{|\wde(t,x)-k|}{\ad(x)}\, dx=\int_{\xdom} \frac{|\wde_0(x)-k|}{\ad(x)}\, dx.
\end{equation}
Similarly, by multiplying with $(\wde-k)^{p-1}$ for $p$ even, we obtain conservation of $L^p$-norms:
\begin{equation}\label{eq:lpcons}
\int_{\xdom} \frac{|\wde(t,x)-k|^p}{\ad(x)}\, dx=\int_{\xdom} \frac{|\wde_0(x)-k|^p}{\ad(x)}\, dx.
\end{equation}
\end{remark}

The moduli of continuity of Lemma \ref{lem:hoeldercont} and Remark \ref{rem:moclp} are independent of $\delta>0$, which implies that the sequence of solutions $\{\wde\}_{\delta>0}$ is relatively compact in $L^p$, $p\in[1,\infty)$ (by Kolmogorov's compactness theorem~\cite[Theorem A.5]{HHNHRise07}), and due to the compact embedding of the H\"older spaces $C^{0,\beta_1}(\Dom)\subset\subset C^{0,\beta_2}(\Dom)$ for $\beta_2<\beta_1$ and bounded domains, also relatively compact in $C^{0,\gamma}(\Dom)$, for any $\gamma<\gamma_\infty$, and thus the limit function $w:=\lim_{\delta\rightarrow 0} \wde\in L^p(\Dom)\cap C^{0,\gamma_\infty}(\Dom)$, $p\in[1,\infty)$ is a weak solution of \eqref{eq:forw} with the same moduli of continuity in space and time. Moreover, the limit $w$ satisfies the entropy identities \eqref{eq:l1cons} and \eqref{eq:lpcons}.
\subsection{Approximation by an upwind scheme}
In order to compute numerical approximations to \eqref{eq:forw}, we
choose $\Delta x>0$ such that $\Nx:=|\xdom|/\Dx\in \N$ and discretize the spatial domain by a grid with
gridpoints $x_{j+1/2}:=j \Delta x$, $j\in \{1,\dots, \Nx\}$ and cell centers $x_{j}:=(j-1/2) \Delta x$, $1\leq j\leq  \Nx$. Furthermore, we
let
\begin{equation}\label{eq:cflcondition}
  0<\Delta t:=\theta \Dx \leq \frac{\Delta x} {\overline{a}}
\end{equation}
and set $t^n:=n\Delta t$, $0\leq n\leq N_T$, where $N_T$ is such that
$t^{N_T}=T$. We define the averaged quantities
\begin{equation}\label{eq:aav}
  a_j=\frac{1}{\Delta x}\int_{x_{j-1/2}}^{x_{j+1/2}} a(x)\, dx,\quad 1\leq j\leq  \Nx,
\end{equation}
and
\begin{equation}\label{eq:initdataapprox}
  w_j^0=\frac{1}{\Delta x}\int_{x_{j-1/2}}^{x_{j+1/2}} w_0(x)\, dx,\quad 1\leq j\leq  \Nx.
\end{equation}
%and finally set $u^0_j:=a_j^{-1} w_j^0$. 
Moreover, we denote, for a quantity $\sigma_j^n$, $j\in \mathbb{Z}$, $n=0,\dots, N_T$ defined on
the grid,
\begin{equation}\label{eq:Dt}
  D^+_t \sigma^n_j:=\frac{1}{\Dt}(\sigma^{n+1}_j-\sigma^n_j), \quad
  D^\pm_x \sigma^n_j=\pm\frac{1}{\Dx}(\sigma^{n}_{j\pm
    1}-\sigma^n_j),\quad D^c_x
  \sigma^n_j=\frac{1}{2\Dx}(\sigma^{n}_{j+1}-\sigma^n_{j-1}). 
\end{equation}
Then we define approximations $w_j^n$ by
\begin{equation}\label{eq:upwindscheme}
  \frac{D^+_t w^{n}_j}{a_j}=-D^-_x w_j^n, \quad 1\le j\le \Nx,\  0\le n\le N_T.
\end{equation}
Letting $u_j^n:=w_j^n/a_j$, this is equivalent to
\begin{equation*}
  D^+_t u_j^{n}=-D^-_x (a_j u_j^n),\quad 1\le j\le \Nx,\ 0\le n\le N_T.
\end{equation*}
which will yield an approximation to the solution $u(t,x)$ of
equation \eqref{eq:setup}.

\subsubsection{Estimates on the numerical approximation}
\label{sssec:estimates}
\begin{lemma}[Properties of the upwind scheme
  \eqref{eq:upwindscheme}]\label{lem:pdisc} 
  The approximations $w_j^n$, $1\le j\le \Nx$, $n=0,\dots, N_T$
  defined by the numerical scheme \eqref{eq:upwindscheme} have the
  following properties:
  \begin{enumerate}
  \item[(i)] Maximum principle:
    \begin{equation}\label{maxprinciple}
      \sup_{1\le j\le \Nx,1\le n\le N_T}{|w^n_j|}\le \sup_{1\le j\le \Nx}{|w_j^0|},
    \end{equation}
  \item[(ii)] Discrete entropy inequality in $L^1$:
    \begin{equation}\label{eq:entropydisc}
      \frac{|w^{n+1}_j- k|}{a_j}- \frac{|w^{n}_j- k|}{a_j}+\frac{\Dt}{\Dx}\bigl(|w^{n}_j- k|-|w^{n}_{j-1}- k|\bigr)\le 0
    \end{equation}
Discrete entropy inequality in $L^2$:
\begin{equation}\label{eq:entropydiscl2}
      \frac{|w^{n+1}_j- k|^2}{a_j}- \frac{|w^{n}_j- k|^2}{a_j}+\frac{\Dt}{\Dx}\bigl(|w^{n}_j- k|^2-|w^{n}_{j-1}- k|^2\bigr)\le 0
    \end{equation}
  \item[(iii)] Bound on the discrete $L^1$ and $L^2$-norms:
    \begin{equation}\label{eq:discl1}
      \Dx\sum_{j=1}^{\Nx} \frac{|w^{n}_j|}{a_j}\le \Dx \sum_{j=1}^{\Nx} \frac{|w^{0}_j|}{a_j},
    \end{equation}
    \begin{equation}\label{eq:L2disc}
      \Dx \sum_{j=1}^{\Nx} \frac{(w^{n}_j)^2}{a_j}\le \Dx \sum_{j=1}^{\Nx} \frac{(w^{0}_j)^2}{a_j}.
    \end{equation}
  \end{enumerate}
  for all $n=1,\dots, N_T$.
\end{lemma}

\begin{proof}
  Writing \eqref{eq:upwindscheme} as
  \begin{equation}\label{eqwithH}
    \frac{w^{n+1}_j}{a_j}=\biggl(1-\frac{a_j\Delta t}{\Delta x}\biggr) \frac{w^{n}_j}{a_j}+\frac{a_j \Delta t}{\Delta x}\frac{w_{j-1}^n}{a_j}:=\mathcal{H}(w^n_{j-1},w^n_j),
  \end{equation}
  and taking the CFL-condition \eqref{eq:cflcondition} into account,
  we immediately see that $w^{n+1}_j/a_j$ is a convex combination of
  $w^n_j/a_j$ and $w_{j-1}^n/a_j$ and thus the approximations satisfy
  the maximum principle \eqref{maxprinciple}.

  To obtain the discrete version \eqref{eq:entropydisc} of a
  Kru\v{z}kov entropy inequality for the quantities $w_j^n$, we
  compute for a constant $k\in \R$,
  \begin{equation}\label{eq:juststh}
    \begin{aligned}
      \mathcal{H}(w^n_{j-1}\vee k,w^n_j\vee
      k)-&\mathcal{H}(w^n_{j-1}\wedge k,w^n_j\wedge k)\\
      &=\biggl(\frac{1}{a_j}-\frac{\Delta t}{\Delta x}\biggr)
      ({w^{n}_j\vee k}-{w^{n}_j\wedge k})+\frac{\Delta t}{\Delta
        x}({w_{j-1}^n\vee k}-w_{j-1}^n\wedge k)\\
      &=\biggl(\frac{1}{a_j}-\frac{\Delta t}{\Delta x}\biggr)
      |w^{n}_j-
      k|+\frac{\Delta t}{\Delta x}|w_{j-1}^n- k|\\
      &=\frac{|w^{n}_j- k|}{a_j}-\frac{\Dt}{\Dt}\bigl(|w^{n}_j-
      k|-|w^{n}_{j-1}- k|\bigr).
    \end{aligned}
  \end{equation}
Moreover, we note that by \eqref{eqwithH} and thanks to the
  CFL-condition \eqref{eq:cflcondition},
  \begin{align*}
    \mathcal{H}(w^n_{j-1}\vee k,w^n_j\vee k)&\ge
    \mathcal{H}(w^n_{j-1},w^n_j)= \frac{w^{n+1}_j}{a_j},\\ 
    \mathcal{H}(w^n_{j-1}\vee k,w^n_j\vee k)&\ge \mathcal{H}(k,k)=
    \frac{k}{a_j},
  \end{align*}
  thus
  \begin{equation*}
    \mathcal{H}(w^n_{j-1}\vee k,w^n_j\vee k)\ge \frac{w^{n+1}_{j}\vee k}{a_j},
  \end{equation*}
  and similarly
  \begin{equation*}
    \mathcal{H}(w^n_{j-1}\wedge k,w^n_j\wedge k)\le \frac{w^{n+1}_{j}\wedge k}{a_j}.
  \end{equation*}
  Combining this with \eqref{eq:juststh}, we obtain
  \eqref{eq:entropydisc}. Now we simply need to multiply the
  expression \eqref{eq:entropydisc} by $\Dx$, sum it over $j=1,\dots, \Nx$ and set $k=0$ to obtain the bound on the discrete
  $L^1(\xdom)$-norm of $w^n_j$,
  \eqref{eq:discl1}. 
To prove the $L^2$-entropy inequality, we note that the difference scheme \eqref{eq:upwindscheme} is equivalent to
\begin{equation}\label{eq:fdmK}
\frac{D^+_t \left(w^n_j-k\right)}{a_j}=-D^-_x\left(w^n_j-k\right)
\end{equation}
for any constant $k\in\R$, and then multiply both sides of equation \eqref{eq:upwindscheme} by $(w^n_j-k)$. Subsequently, we use that
\begin{equation}\label{eq:binom}
    ab=\frac{1}{2}(a^2+b^2-(a-b)^2),\quad a,b\in \R,
  \end{equation}
  once for the left hand side and once for the right hand side to get (we write $\wh^n_j:=w^n_j-k$ for convenience)
  \begin{equation*}
    \frac{1}{2 a_j}\bigl(
    (\wh^{n+1}_j)^2+(\wh_j^n)^2-(\wh^{n+1}_j-\wh^{n}_j)^2\bigr)=\frac{(\wh^n_j)^2}{a_j}
    +\frac{\Dt}{2
      \Dx}\bigl( -(\wh^n_j)^2+(\wh^n_{j-1})^2-(\wh^n_j-\wh^n_{j-1})^2\bigl). 
  \end{equation*}
  Rearranging terms and using \eqref{eq:fdmK} for the
  difference $\wh^{n+1}_j-\wh^{n}_j$, this reads
  \begin{equation*}
    \frac{(\wh^{n+1}_j)^2}{2 a_j}=\frac{(\wh^n_j)^2}{2 a_j}+\frac{\Dt}{2
      \Dx}\bigl( -(\wh^n_j)^2+(\wh^n_{j-1})^2\bigr)+\frac{\Dt}{2
      \Dx}\biggl(\frac{a_j \Dt}{\Dx}-1\biggr)(\wh^n_j-\wh^n_{j-1})^2. 
  \end{equation*}
from which the claim follows using the CFL-condition  \eqref{eq:cflcondition}. Summing the discrete $L^2$-inequality over $j$ and using induction over $n$, we furthermore obtain \eqref{eq:L2disc}.

\end{proof}
Thanks to the linearity of the transport equation, we obtain the
following corollary of Lemma~\ref{lem:pdisc}:
\begin{corollary}\label{cor:contraction}
  Let $w^n_j$ denote the approximations computed by the scheme
  \eqref{eq:upwindscheme}, \eqref{eq:initdataapprox} for initial
  data $w_0\in L^1(\xdom)\cap L^2(\xdom)\cap
  L^\infty(\overline{\xdom})$ and $v_j^n$ another approximation computed by
  \eqref{eq:upwindscheme} for initial data $v_0\in
  L^1(\xdom)\cap L^2(\xdom)\cap L^\infty(\overline{\xdom})$. Then
  we have
  \begin{equation}
    \label{eq:differences}
    \begin{split}
      \sup_{1\le j\le \Nx}{|w^n_j-v_j^n|}&\le \sup_{1\le j\le \Nx}{|w_j^0-v_j^0|}\le \|w_0-v_0\|_{\infty},\\
      \Dx \sum_j a_j^{-1}|w_j^n-v_j^n|&\le \Dx \sum_j a_j^{-1}|w^0_j-v_j^0|\le\|(w_0-v_0)/a\|_{L^1(\xdom)},\\
      \Dx \sum_j a_j^{-1}|w_j^n-v_j^n|^2&\le \Dx \sum_j
      a_j^{-1}|w_j^0-v_j^0|^2\le\|(w_0-v_0)/\sqrt{a}\|^2_{L^2(\xdom)},
    \end{split}
  \end{equation}
  for all $1\le n\le N_T$.
\end{corollary}
\begin{proof}
  This follows from the fact that the differences $r_j^n:=w_j^n-v_j^n$
  satisfy \eqref{eq:upwindscheme} due to the linearity of the scheme,
  together with Lemma \ref{lem:pdisc}.
\end{proof}
Defining the piecewise constant approximations
\begin{equation}\label{eq:defapproxfcn}
  w_{\Delta x}(t,x):=w^n_j,\quad (t,x)\in [t^n,t^{n+1})\times [x_{j-1/2},x_{j+1/2}),
\end{equation}
this corollary enables us to show that the piecewise constant function $w_{\Dx}$ has a modulus of continuity in time:
\begin{lemma}\label{lem:mocdisc}
The piecewise constant functions $w_{\Dx}$ defined in \eqref{eq:defapproxfcn} have a modulus of continuity in time if the initial data $w_0$ satisfies \eqref{eq:mocinit} for $p=1$: 
\begin{equation}\label{eq:moc1disc}
\sup_{|h|\leq \sigma} \int_{\xdom} |w_{\Dx}(t+h,x)-w_{\Dx}(t,x)|\, dx\leq C (\sigma+\Dx)^{\gamma_1}
\end{equation}
or \eqref{eq:mocinit} for $p=2$:
\begin{equation}
\sup_{|h|\leq \sigma} \int_{\xdom} |w_{\Dx}(t+h,x)-w_{\Dx}(t,x)|^2\, dx\leq C (\sigma+\Dx)^{2\gamma_2}.
\end{equation}
If the initial data is H\"{o}lder continuous with exponent $\gamma_\infty$, the solution is approximately H\"{o}lder continuous in time with the same exponent, i.e. 
\begin{equation*}
\sup_{|h|\leq\sigma} \|w_{\Dx}(t+h,x)-w_{\Dx}(t,x)\|_{L^\infty}\leq C (\sigma+\Dx)^{\gamma_\infty}.
\end{equation*}

\end{lemma}
\begin{proof}
We observe that 
\begin{align*}
w^k_j&=\left(1-\lambda a_j\right) w^{k-1}_j+\lambda a_j w^{k-1}_{j-1}\\
&= \left(1-\lambda a_j\right)\left\{ \left(1-\lambda a_j\right) w^{k-2}_j+\lambda a_j w^{k-2}_{j-1}\right\}+\lambda a_j \left\{\left(1-\lambda a_{j-1}\right) w^{k-2}_{j-1}+\lambda a_{j-1}w^{k-2}_{j-2}\right\} \\
&= \dots = \sum_{\ell=j-k}^j \lambda_{\ell,j}^k w^0_{\ell}
\end{align*}
where we have denoted $\lambda:=\Dt/\Dx$ and $\lambda_{\ell,j}^k$ is inductively defined as follows:
\begin{definition}\label{def:lambda}
We let $s_\ell:=\lambda a_\ell$ and $v_\ell=(1-\lambda a_\ell)$ for all $\ell\in\Z$. Then we define $\lambda_{\ell,j}^k$ recursively by
\begin{align}\label{eq:lambdadef}
\begin{split}
\lambda_{\ell,j}^k=\begin{cases}
1,& \ell=j,\, k=0;\\
0, &k+\ell<j\, \mathrm{or}\, \ell>j,\\
s_{\ell}\lambda_{\ell,j-1}^{k-1}+v_{\ell} \lambda_{\ell,j}^{k-1}; &\mathrm{otherwise}.
\end{cases}
\end{split}
\end{align}
\end{definition}
\begin{remark}
	An explicit expression for the coefficients $\lambda_{\ell,j}^k$ for the third case is given by
	\begin{equation*}
	\lambda_{j-\ell,j}^k=\prod_{n=j-\ell+1}^j s_n\sum_{\substack{m_i\geq 0,\\ \sum_{i=1}^{l+1}m_i=k-\ell}} v_j^{m_1}\cdot . . . \cdot  v_{j-\ell}^{m_{\ell+1}}.
	\end{equation*}
\end{remark}

\begin{claim}\label{cl:bla}
The coefficients $\lambda_{\ell,j}^k$ satisfy
\begin{equation}\label{eq:sumlambda}
\sum_{\ell=j-k}^k \lambda_{\ell,j}^k =1
\end{equation}
and $\lambda_{\ell,j}^k\geq 0$.
\end{claim}
\begin{proof}[Proof of Claim]
Equation \eqref{eq:sumlambda} follows by induction and using that $v_\ell+s_\ell=1$ for any $\ell\in\Z$. That the $\lambda_{\ell,j}^k$ are nonnegative follows from the CFL-condition \eqref{eq:cflcondition}.
\end{proof}
Thus we have, 
\begin{align*}
\int_{\xdom} |w_{\Dx}(k\Dt,x)-w_{\Dx}(0,x)| dx &= \Dx \sum_j |w^k_j -w^0_j|\\
&=\Dx\sum_j \bigg|\sum_{\ell=j-k}^j \lambda^k_{\ell,j} w^0_\ell -w^0_j\bigg|\\
&\leq \Dx\sum_j \sum_{\ell=j-k}^j \lambda^k_{\ell,j}| w^0_\ell -w^0_j|\\
&\leq \Dx \sum_j \sum_{\ell=j-k}^j \lambda^k_{\ell,j}\max_{j-k\leq m\leq j} |w^0_m -w^0_j|\\
&=\Dx \sum_j \max_{j-k\leq m\leq j} |w^0_m -w^0_j|\\
&\leq \int_{\xdom} \sup_{|h|\leq (k+1)\Dx} |w_0(x+h)-w_0(x)|\, dx\\
&\leq C (\Dx k)^{\gamma_1}
\end{align*}
where we used Claim \ref{cl:bla} for the first inequality and the third equality, and the assumption on the initial data, \eqref{eq:mocinit}, in the last inequality.
Similarly, we compute in the $L^2$-setting,
\begin{align*}
\int_{\xdom} |w_{\Dx}(k\Dt,x)-w_{\Dx}(0,x)|^2 dx &= \Dx \sum_j |w^k_j -w^0_j|^2\\
&=\Dx\sum_j \bigg|\sum_{\ell=j-k}^j \lambda^k_{\ell,j} w^0_\ell -w^0_j\bigg|^2\\
&\leq \Dx\sum_j \sum_{\ell_1=j-k}^j \sum_{\ell_2=j-k}^j\lambda^k_{\ell_1,j}\lambda^k_{\ell_2,j}| w^0_{\ell_1} -w^0_j|| w^0_{\ell_2} -w^0_j|\\
&\leq \Dx \sum_j \sum_{\ell_1=j-k}^j \sum_{\ell_2=j-k}^j\lambda^k_{\ell_1,j}\lambda^k_{\ell_2,j}\max_{j-k\leq m\leq j} |w^0_m -w^0_j|^2\\
&=\Dx \sum_j \max_{j-k\leq m\leq j} |w^0_m -w^0_j|^2\\
&\leq \int_{\xdom} \sup_{|h|\leq (k+1)\Dx} |w_0(x+h)-w_0(x)|^2\, dx\\
&\leq C (\Dx k)^{2\gamma_2}
\end{align*}
Then applying Corollary \ref{cor:contraction}, we conclude. The approximate H\"{o}lder continuity in time follows in a very similar way, instead of summing over $j$, we take the maximum over all $j$.
\end{proof}

\begin{remark}
If $w_0$ has bounded variation, obtaining that the solution has bounded variation and so $L^1$-moduli of continuity in time and space is much easier. Similarly, if the initial data has a modulus of continuity of $\gamma_2=1$ in $L^2$, obtaining a rate is easier. The argument is similar to the one which will be outlined for the linear wave equation in Section \ref{sec:1d} but less technical.
\end{remark}

\subsection{A convergence rate in $L^1$}\label{ssec:l1rate}
In this section, we will prove a rate of convergence in $L^1$ of the numerical scheme to the limit of solutions of \eqref{eq:forwdelta} as $\delta\rightarrow 0$. Our approach is based on the doubling of variables technique developed by S.~N.~Kru\v{z}kov~\cite{k1970} for nonlinear scalar conservation laws. The reason why we take this approach is the possible nondifferentiability of the solutions with which the doubling of variables technique can deal well.
For this purpose, we use the test function $\omega_{\delta}$ from \eqref{eq:omegadelta}, let $0<\nu<\tau<T$ and $\epst,\epsx>0$ such that $0<2\epst<\min\{\nu,T-\tau\}$ and $\Dt,\Dx<\min\{\epst,\epsx\}$ and define the function $\Omega:\Dom^2\rightarrow \R$ by
\begin{equation}\label{eq:testfcn1}
  \Omega(t,s,x,y)=\mathbf{1}_{[\nu,\tau)}(t)\omega_{\epst}(t-s)\omega_{\epsx}(x-y).
\end{equation}
We note that the (smooth) solution to \eqref{eq:forwdelta} satisfies
\begin{equation}\label{eq:kruz1}
\int_{\Dom} \frac{|\wde(s,y)-k|}{\ad(y)}\Omega_s+|\wde(s,y)-k|\Omega_y \,dyds=0,
\end{equation}
whereas the approximations $\wdi$ satisfy, by the discrete entropy inequality \eqref{eq:entropydisc}
\begin{equation}\label{eq:kruz2}
\int_{\Dom} \frac{|\wdi(t,x)-\ell|}{a_{\Dx}(x)}D^-_t\Omega+|\wdi(t,x)-\ell|D^+_x\Omega \,dxdt\geq 0,
\end{equation}
where we have denoted
\begin{equation*}
    a_\Dx(x)=a_j, \quad x\in [x_{j-1/2},x_{j+1/2}).
  \end{equation*}

\begin{theorem}\label{lem:ratetransport}
  Let $a\in C^{0,\alpha}(\xdom)$ satisfy
  \eqref{eq:ascona}. Denote $w:=\lim_{\delta\rightarrow 0}w^\delta$ the solution of \eqref{eq:forw} and
  $w_\Dx$ the numerical approximation computed by scheme
  \eqref{eq:upwindscheme} and defined in
  \eqref{eq:defapproxfcn}. Assume that the initial data $w_0\in
  L^1(\xdom)$ and is H\"{o}lder continuous with exponent $\gamma_{\infty}>0$. 
  Then $w_\Dx(t,\cdot)$ converges to the solution $w(t,\cdot)$,
  $0<t<T$, at (at least) the rate
 \begin{equation}\label{eq:rae1d}
    \|(w-w_{\Dx})(t,\cdot)\|_{L^1(\xdom)}\le C\, \Dx^{(\gamma_\infty\alpha)/(\gamma_\infty\alpha+2-\gamma_\infty)}+C\|(w_0-w_\Dx(0,\cdot))\|_{L^1(\xdom)},
  \end{equation}
   where $C$ is a constant depending on $\ua$,
  $\oa$, $\|a\|_{C^{0,\alpha}}$ and $T$ but not on $\Dx$.
\end{theorem}
\begin{proof}

Inserting $\wdi(t,x)$ for $k$ and $\wde(s,y)$ for $\ell$ in \eqref{eq:kruz1} and \eqref{eq:kruz2}, and integrating the respective equations over $(t,x)\in \Dom$ and $(s,y)\in \Dom$ respectively, and adding up, we have (for convenience, we will omit writing the arguments of $\wdi=\wdi(t,x)$, $\wde=\wdi(s,y)$, $\ad=\ad(y)$ and $a_{\Dx}=a_{\Dx}(x)$ in the following)
\begin{equation}\label{eq:figdi}
\int_{\Dom^2} {|\wdi-\wde|}\left(\frac{D^-_t \Omega}{a_{\Dx}}+\frac{\Omega_s}{\ad}\right)+|\wdi-\wde|\left(D^+_x\Omega +\Omega_y\right)\,\duz\geq 0,
\end{equation}
where $\duz:=dx dy dt ds$.
We have
\begin{align}\label{eq:Domega}
\begin{split}
D^-_t\Omega&=D^-_t \mathbf{1}_{[\nu,\tau)}(t) \omega_{\epst}(t-\Dt-s)+ \mathbf{1}_{[\nu,\tau)}(t) D^-_t\omega_{\epst}(t-s)\\
&=D^-_t \mathbf{1}_{[\nu,\tau)}(t) \omega_{\epst}(t-\Dt-s)- \mathbf{1}_{[\nu,\tau)}(t) D^+_s\omega_{\epst}(t-s),
\end{split}
\end{align}
so that we can rewrite equation \eqref{eq:figdi} as
\begin{align}\label{eq:figdi2}
\begin{split}
&A+B+D+E\\
&\quad :=\int_{\Dom^2} \frac{|\wdi-\wde|}{a_{\Dx}}\mathbf{1}_{[\nu,\tau)}(t)\omega_{\epsx}\left(\partial_s\omega_{\epst}-D^+_s\omega_{\epst}\right)\, \duz\\
 &\quad\quad+ \int_{\Dom^2} \frac{|\wdi-\wde|}{a_{\Dx}\ad}\mathbf{1}_{[\nu,\tau)}(t)\omega_{\epsx} \partial_s \omega_{\epst}\left(a_{\Dx}-\ad\right)\, \duz\\
&\quad\quad\quad+\int_{\Dom^2} |\wdi-\wde|\mathbf{1}_{[\nu,\tau)}(t)\omega_{\epst}\left(\partial_y\omega_{\epsx}+ D^+_x\omega_{\epsx}\right)\, \duz \\
&\quad\quad\quad\quad+\int_{\Dom^2} \frac{|\wdi-\wde|}{a_{\Dx}}\omega_{\epsx}\omega_{\epst}(t-s-\Dt) D^-_t\mathbf{1}_{[\nu,\tau)}(t)\, \duz\\
&\quad\quad\quad\quad\quad\geq 0,
\end{split}
\end{align}
We note that
\begin{equation}\label{eq:Domega2}
D^-_t \mathbf{1}_{[\nu,\tau)}(t)=\frac{1}{\Dt}\mathbf{1}_{[\nu,\nu+\Dt)}(t)-\frac{1}{\Dt}\mathbf{1}_{[\tau,\tau+\Dt)}(t),
\end{equation}
which means that we can rewrite \eqref{eq:figdi2} as
\begin{multline}\label{eq:schissi1}
\frac{1}{\Dt}\int_{\Dom^2} \frac{|\wdi-\wde|}{a_{\Dx}}\omega_{\epsx}\omega_{\epst}(t-s-\Dt) \mathbf{1}_{[\tau,\tau+\Dt)}(t)\, \duz\\
\leq \frac{1}{\Dt}\int_{\Dom^2} \frac{|\wdi-\wde|}{a_{\Dx}}\omega_{\epsx}\omega_{\epst}(t-s-\Dt) \mathbf{1}_{[\nu,\nu+\Dt)}(t)\, \duz+A+B+D
\end{multline}
We begin by estimating the term $A$. To do so, we note that
\begin{equation}\label{eq:omegas}
\partial_s\omega_{\epst}- D^+_s \omega_{\epst}=\frac{1}{\Dt}\int_{0}^{\Dt} (\xi-\Dt)\partial_{ss}\omega_{\epst}(t-s+\xi)\, d\xi
\end{equation}
and that
\begin{equation}\label{eq:null0}
\frac{1}{\Dt}\int_0^{\Dt}\int_{\Dom^2} \frac{|\wdi-\wde(t,y)|}{a_{\Dx}}\mathbf{1}_{[\nu,\tau)}(t)\omega_{\epsx}(\xi-\Dt)\partial_{ss}\omega_{\epst}(t-s+\xi)\, \duz\, d\xi=0,
\end{equation}
because 
$\frac{|\wdi-\wde(t,y)|}{a_{\Dx}}\mathbf{1}_{[\nu,\tau)}(t)\omega_{\epsx}(\xi-\Dt)$ is constant with respect to $s$ and $\omega_{\epst}(t+\xi-\cdot)$ is compactly supported in the domain. Therefore we can rewrite the term $A$ as
\begin{equation}\label{eq:Aa}
A=\frac{1}{\Dt}\int_0^{\Dt}\int_{\Dom^2} \frac{(|\wdi-\wde|-|\wdi-\wde(t,y)|)}{a_{\Dx}}\mathbf{1}_{[\nu,\tau)}(t)\omega_{\epsx}(\xi-\Dt)\partial_{ss}\omega_{\epst}(t-s+\xi)\, \duz\, d\xi,
\end{equation}
and bound in the following way, using triangle inequality,
\begin{align}\label{eq:A}
\begin{split}
|A|&\leq \frac{1}{\Dt}\int_0^{\Dt}\int_{\Dom^2} \frac{|\wde-\wde(t,y)|}{a_{\Dx}}\mathbf{1}_{[\nu,\tau)}(t)\omega_{\epsx}|\xi-\Dt| \big|\partial_{ss}\omega_{\epst}(t-s+\xi)\big|\, \duz\, d\xi,\\
& \leq  \frac{1}{\Dt\ua}\int_0^{\Dt} \int_{\nu}^{\tau}\int_{\Dom} |\wde-\wde(t,y)|\, dy\, |\xi-\Dt| \big|\partial_{ss}\omega_{\epst}(t-s+\xi)\big|\, dsdtd\xi\\
&\leq \frac{1}{\Dt\ua}\int_0^{\Dt} |\xi-\Dt|\int_{\nu}^{\tau}\sup_{s\in [t-\epst,t+2\epst]} \int_{\xdom} |\wde-\wde(t,y)|\, dy \int_{0}^T \big|\partial_{ss}\omega_{\epst}(t-s+\xi)\big|\, dsdtd\xi\\
&\leq \frac{C\Dt }{\ua\,  \epst^2} \int_{\nu}^{\tau}\sup_{s\in [t-\epst,t+2\epst]} \int_{\xdom} |\wde-\wde(t,y)|\, dy \, dt\\
&\leq \frac{C\Dt T}{\ua  \,\epst^{2-\gamma_\infty}}
\end{split}
\end{align}
where we have used Lemma \ref{lem:hoeldercont} for the last inequality. We proceed to estimating the term $B$:
\begin{equation*}
B:=\int_{\Dom^2} \frac{|\wdi-\wde|}{a_{\Dx}\ad}\mathbf{1}_{[\nu,\tau)}(t)\omega_{\epsx} \partial_s \omega_{\epst}\left(a_{\Dx}-\ad\right)\, \duz
\end{equation*}
Similarly to the case of term $A$, we use that
\begin{equation*}
\int_{\Dom^2} \frac{|\wdi-\wde(t,y)|}{a_{\Dx}\ad}\mathbf{1}_{[\nu,\tau)}(t)\omega_{\epsx} \partial_s \omega_{\epst}\left(a_{\Dx}-\ad\right)\, \duz=0.
\end{equation*}
Thus subtracting this from $B$, we can bound term $B$ using triangle inequality,
\begin{align}\label{eq:B}
\begin{split}
|B|&\leq \int_{\Dom^2} \frac{|\wde-\wde(t,y)|}{a_{\Dx}\ad}\mathbf{1}_{[\nu,\tau)}(t)\omega_{\epsx} |\partial_s \omega_{\epst}|\big|a_{\Dx}-\ad\big|\, \duz\\
&\leq \frac{1}{\ua^2} \int_{\Dom^2} |\wde-\wde(t,y)|\mathbf{1}_{[\nu,\tau)}(t)\omega_{\epsx} |\partial_s \omega_{\epst}|\left(\big|a_{\Dx}-\ad(x)\big|+\big|\ad(x)-\ad\big|\right)\, \duz\\
&\leq \frac{\|a_{\Dx}-\ad\|_{\infty}+ \sup_{|x-y|\leq 2\epsx}|\ad(x)-\ad(y)|}{\ua^2}\int_{\Dom^2} |\wde-\wde(t,y)|\mathbf{1}_{[\nu,\tau)}(t)\omega_{\epsx} |\partial_s \omega_{\epst}|\, \duz\\
&\leq C\frac{\Dx^{\alpha}+\delta^{\alpha}+\epsx^{\alpha}}{\ua^2}\int_{\nu}^{\tau}\sup_{s\in [t-\epst,t+\epst]}\int_{\xdom}|\wde(s,y)-\wde(t,y)|\, dy \int_0^T |\partial_s \omega_{\epst}| ds dt\\
&\leq C \frac{\Dx^{\alpha}+\delta^{\alpha}+\epsx^{\alpha}}{\ua^2} \epst^{\gamma_\infty-1}\\
&\leq C \frac{\delta^\alpha+\epsx^\alpha}{\epst^{1-\gamma_\infty}},
\end{split}
\end{align}
since we assumed $\Dx<\epsx$. We continue to bound term $D$. First, we observe that
\begin{equation}\label{eq:weissnoed}
\partial_y\omega_{\epsx}+ D^+_x\omega_{\epsx}=\partial_y\omega_{\epsx}- D^-_y\omega_{\epsx}=\frac{1}{\Dx}\int_0^{\Dx}(\Dx-\xi) \partial_{yy}\omega_{\epsx}(x-y+\xi)\, d\xi
\end{equation}
and that
\begin{equation*}
\frac{1}{\Dx}\int_0^{\Dx}(\Dx-\xi)\int_{\Dom^2} |\wdi-\wde(s,x)|\mathbf{1}_{[\nu,\tau)}(t)\omega_{\epst}\partial_{yy}\omega_{\epsx}(x-y+\xi)\, \duz\, d\xi=0,
\end{equation*}
since $\omega_{\epsx}$ is compactly supported. Thus, we can estimate $D$ by
\begin{align}\label{eq:D}
\begin{split}
|D|&\leq \frac{1}{\Dx}\int_0^{\Dx}|\Dx-\xi|\int_{\Dom^2} |\wde-\wde(s,x)|\mathbf{1}_{[\nu,\tau)}(t)\omega_{\epst}\big|\partial_{yy}\omega_{\epsx}(x-y+\xi)\big|\, \duz\, d\xi\\
&\leq \frac{1}{\Dx}\int_0^{\Dx}|\Dx-\xi|\int_{\nu}^{\tau}\int_0^T\sup_{|h|\leq 2\epsx}\int_{\xdom} |\wde(s,y)-\wde(s,y+h)|\omega_{\epst}\int_{\xdom}\big|\partial_{yy}\omega_{\epsx}\big|\,dy \, dxdsdt\, d\xi\\
&\leq C\frac{\Dx}{\epsx^2}\int_{\nu}^{\tau}\int_0^T\sup_{|h|\leq 2\epsx}\int_{\xdom} |\wde(s,y)-\wde(s,y+h)|\, dx\omega_{\epst}\, dsdt\\
&\leq C\frac{\Dx\, T}{\epsx^{2-\gamma_\infty}}.
\end{split}
\end{align}
It remains to relate the terms forming $E$ to the $L^1$-norm of the differences $\wde(t,\cdot)-\wdi(t,\cdot)$. We have
\begin{align*}
E_2&:=\frac{1}{\Dt}\int_{\Dom^2} \frac{|\wdi-\wde|}{a_{\Dx}}\omega_{\epsx}\omega_{\epst}(t-s-\Dt) \mathbf{1}_{[\tau,\tau+\Dt)}(t)\, \duz\\
&=\frac{1}{\Dt}\int_{\tau}^{\tau+\Dt}\int_0^T\int_{\xdom^2} \frac{|\wdi-\wde|}{a_{\Dx}}\omega_{\epsx}\, dx dy\,\omega_{\epst}(t-s-\Dt) \, dsdt\\
\end{align*}
and can rewrite
\begin{equation*}
\|(\wdi-\wde)/a_{\Dx}\|_{L^1(\xdom)}(\tau)=\frac{1}{\Dt}\int_{\Dom^2}\frac{|\wdi(\tau,x)-\wde(\tau,x)|}{a_{\Dx}}\omega_{\epsx}\omega_{\epst}(t-s-\Dt) \mathbf{1}_{[\tau,\tau+\Dt)}(t)\, \duz,
\end{equation*}
so that
\begin{align}\label{eq:E_2}
\begin{split}
&\Big| E_2-\|(\wdi-\wde)/a_{\Dx}\|_{L^1(\xdom)}(\tau)\Big|\\
&\quad =\bigg|\frac{1}{\Dt}\int_{\Dom^2}\omega_{\epsx}\mathbf{1}_{[\tau,\tau+\Dt)}(t)\frac{1}{a_{\Dx}} \left(|\wdi-\wde|-|\wdi(\tau,x)-\wde(\tau,x)|\right)\omega_{\epst}(t-s-\Dt) \, \duz\bigg|\\
&\quad\leq \frac{1}{\Dt\ua}\int_{\Dom^2}\omega_{\epsx}\mathbf{1}_{[\tau,\tau+\Dt)}(t)\omega_{\epst}(t-s-\Dt)\left(|\wdi-\wdi(\tau,\cdot)|+|\wde-\wde(\tau,x)|\right)\, \duz\\
&\quad\leq \frac{1}{\Dt\ua}\int_{[0,T]^2}\mathbf{1}_{[\tau,\tau+\Dt)}(t)\omega_{\epst}(t-s-\Dt)\int_{\xdom}|\wdi-\wdi(\tau,\cdot)|\, dx\, \duz\\
&\qquad + \frac{1}{\Dt\ua}\int_{\Dom^2}\omega_{\epsx}\mathbf{1}_{[\tau,\tau+\Dt)}(t)\omega_{\epst}(t-s-\Dt)\left(|\wde-\wde(\tau,y)|+|\wde(\tau,y)-\wde(\tau,x)|\right)\, \duz\\
&\quad\leq \frac{1}{\Dt\ua}\int_{\tau}^{\tau+\Dt}\sup_{|h|\leq 2\epst}\int_{\xdom}|\wdi-\wdi(t+h,\cdot)|\, dx\,dt\\
&\qquad + \frac{1}{\Dt\ua}\int_{[0,T]^2}\mathbf{1}_{[\tau,\tau+\Dt)}(t)\omega_{\epst}(t-s-\Dt)\int_{\xdom}|\wde-\wde(\tau,\cdot)|\, dx \,dsdt\\
&\qquad + \frac{1}{\Dt\ua}\int_{\xdom^2}\omega_{\epsx}|\wde(\tau,y)-\wde(\tau,x)|\, dxdy\\
&\quad \leq C \epst^{\gamma_\infty}+C \epsx^{\gamma_\infty}
\end{split}
\end{align}
In a similar way, defining
\begin{align*}
E_1&:=\frac{1}{\Dt}\int_{\Dom^2} \frac{|\wdi-\wde|}{a_{\Dx}}\omega_{\epsx}\omega_{\epst}(t-s-\Dt) \mathbf{1}_{[\nu,\nu+\Dt)}(t)\, \duz\\
&=\frac{1}{\Dt}\int_{\nu}^{\nu+\Dt}\int_0^T\int_{\xdom^2} \frac{|\wdi-\wde|}{a_{\Dx}}\omega_{\epsx}\, dx dy\,\omega_{\epst}(t-s-\Dt) \, dsdt
\end{align*}
we obtain
\begin{equation}\label{eq:E_1}
\bigg| E_1-\|(\wde_0-\wdi(0,\cdot))/a_{\Dx}\|_{L^1}\bigg|\leq C (\epsx^{\gamma_\infty}+\epst^{\gamma_\infty}+\nu^{\gamma_\infty})
\end{equation}
Thus, combining the estimates \eqref{eq:figdi2}, \eqref{eq:A}, \eqref{eq:B}, \eqref{eq:D}, \eqref{eq:E_2} and \eqref{eq:E_1}, 
we have
\begin{multline*}
\|\wdi-\wde\|_{L^1}(\tau) \\ \leq C \|\wde_0-\wdi(0,\cdot)\|_{L^1} +C\left(\epsx^{\gamma_\infty}+\epst^{\gamma_\infty}+\nu^{\gamma_\infty} +\Dt\epst^{\gamma_\infty-2}+ (\delta^\alpha+\epsx^\alpha)\epst^{\gamma_\infty-1}+\Dx\epsx^{\gamma_\infty-2}\right)
\end{multline*}
We let $\delta\rightarrow 0$,
\begin{equation}\label{eq:L1pre}
\|\wdi-w\|_{L^1}(\tau)\leq C \left(\|w_0-\wdi(0,\cdot)\|_{L^1} + \epsx^{\gamma_\infty}+\epst^{\gamma_\infty}+\nu^{\gamma_\infty} +\Dt\epst^{\gamma_\infty-2}+ \epsx^\alpha\epst^{\gamma_\infty-1}+\Dx\epsx^{\gamma_\infty-2}\right)
\end{equation}
We choose in this last expression $\nu=3 \epst$, $\epsx=\epst^{1/\alpha}$ and $\epsx=\Dx^{1/(\alpha\gamma_\infty+2-\gamma_\infty)}$ to obtain the rate.
\end{proof}
\begin{remark}
	Note that the above lemma implies a rate of convergence in $L^1(\xdom)$ of at least $\min\{\alpha, \nicefrac{(\gamma_\infty\alpha)}{(\gamma_\infty\alpha+2-\gamma_\infty)}\}$ for the variable $u_{\Dx}=w_{\Dx}/a_{\Dx}$.
\end{remark}

\subsection{A convergence rate in $L^2$}\label{sec:l2acvection}
The main ideas for proving a rate of convergence in $L^2$ are similar to those in Section \ref{ssec:l1rate}, an additional tool involved is a type of Gr\"{o}nwall inequality and the whole procedure is a bit more technical.
We start by noting that the (smooth) solution to \eqref{eq:forwdelta} satisfies
\begin{equation}\label{eq:kruz3}
\int_{\Dom} \frac{|\wde(s,y)-k|^2}{\ad(y)}\Omega_s+|\wde(s,y)-k|^2\Omega_y \,dyds=0,
\end{equation}
whereas the approximations $\wdi$ satisfy, by the discrete entropy inequality, \eqref{eq:entropydiscl2}
\begin{equation}\label{eq:kruz4}
\int_{\Dom} \frac{|\wdi(t,x)-\ell|^2}{a_{\Dx}(x)}D^-_t\Omega+|\wdi(t,x)-\ell|^2 D^+_x\Omega \,dxdt\geq 0.
\end{equation}
Then we can prove
\begin{theorem}\label{lem:ratetransportl2}
  Let $a\in C^{0,\alpha}(\R)$ satisfy
  \eqref{eq:ascona}. Denote $w$ the solution of \eqref{eq:forw} obtained as the limit as $\delta\rightarrow 0$ in \eqref{eq:forwdelta} and
  $w_\Dx$ the numerical approximation computed by scheme
  \eqref{eq:upwindscheme} and defined in
  \eqref{eq:defapproxfcn}. Assume that the initial data $w_0\in
  L^1(\xdom)$ is H\"{o}lder continuous with exponent $\gamma_\infty>0$. Then $w_\Dx(t,\cdot)$ converges to the solution $w(t,\cdot)$,
  $0<t<T$, at (at least) the rate
  \begin{equation}\label{eq:rae1dl2}
    \|(w-w_{\Dx})(\tau,\cdot)\|_{L^2(\xdom)}\le C\, \Dx^{(\gamma_\infty\alpha)/(\gamma_\infty\alpha+2-\gamma_\infty)}+C\|w_0-w_\Dx(0,\cdot)\|_{L^2(\xdom)},
  \end{equation}
  where $C$ is a constant depending on $\ua$,
  $\oa$, $\|a\|_{C^{0,\alpha}}$ and $T$ but not on $\Dx$.
\end{theorem}

\begin{proof}
As in the $L^1$-case, we insert $\wdi(t,x)$ for $k$ and $\wde(s,y)$ for $\ell$ in \eqref{eq:kruz3} and \eqref{eq:kruz4}, and integrate the respective equations over $(t,x)\in\Dom$ and $(s,y)\in \Dom$ respectively. Then adding up, we have (for convenience, we will again omit writing the arguments of $\wdi=\wdi(t,x)$, $\wde=\wdi(s,y)$, $\ad=\ad(y)$ and $a_{\Dx}=a_{\Dx}(x)$ in the following)
\begin{equation}\label{eq:figdi3}
\int_{\Dom^2} {|\wdi-\wde|^2}\left(\frac{D^-_t \Omega}{a_{\Dx}}+\frac{\Omega_s}{\ad}\right)+|\wdi-\wde|^2\left(D^+_x\Omega +\Omega_y\right)\,\duz\geq 0,
\end{equation}
where $\duz:=dx dy dt ds$. By \eqref{eq:Domega}, we can rewrite equation \eqref{eq:figdi3} as
\begin{align}\label{eq:figdi4}
\begin{split}
&A+B+D+E\\
&\quad :=\int_{\Dom^2} \frac{|\wdi-\wde|^2}{a_{\Dx}}\mathbf{1}_{[\nu,\tau)}(t)\omega_{\epsx}\left(\partial_s\omega_{\epst}-D^+_s\omega_{\epst}\right)\, \duz\\
 &\quad\quad+ \int_{\Dom^2} \frac{|\wdi-\wde|^2}{a_{\Dx}\ad}\mathbf{1}_{[\nu,\tau)}(t)\omega_{\epsx} \partial_s \omega_{\epst}\left(a_{\Dx}-\ad\right)\, \duz\\
&\quad\quad\quad+\int_{\Dom^2} |\wdi-\wde|^2\mathbf{1}_{[\nu,\tau)}(t)\omega_{\epst}\left(\partial_y\omega_{\epsx}+ D^+_x\omega_{\epsx}\right)\, \duz \\
&\quad\quad\quad\quad+\int_{\Dom^2} \frac{|\wdi-\wde|^2}{a_{\Dx}}\omega_{\epsx}\omega_{\epst}(t-s-\Dt) D^-_t\mathbf{1}_{[\nu,\tau)}(t)\, \duz\\
&\quad\quad\quad\quad\quad\geq 0,
\end{split}
\end{align}
and by \eqref{eq:Domega2}, this is equivalent to
\begin{multline}\label{eq:schissix}
\frac{1}{\Dt}\int_{\Dom^2} \frac{|\wdi-\wde|^2}{a_{\Dx}}\omega_{\epsx}\omega_{\epst}(t-s-\Dt) \mathbf{1}_{[\tau,\tau+\Dt)}(t)\, \duz\\
\leq \frac{1}{\Dt}\int_{\Dom^2} \frac{|\wdi-\wde|^2}{a_{\Dx}}\omega_{\epsx}\omega_{\epst}(t-s-\Dt) \mathbf{1}_{[\nu,\nu+\Dt)}(t)\, \duz+A+B+D
\end{multline}
We start with the term $A$. To do so, we use again \eqref{eq:omegas} and that, similarly to \eqref{eq:null0}, it holds,
\begin{equation}\label{eq:null01}
\frac{1}{\Dt}\int_0^{\Dt}\int_{\Dom^2} \frac{|\wdi-\wde(t,y)|^2}{a_{\Dx}}\mathbf{1}_{[\nu,\tau)}(t)\omega_{\epsx}(\xi-\Dt)\partial_{ss}\omega_{\epst}(t-s+\xi)\, \duz\, d\xi=0,
\end{equation}
Hence we can rewrite the term $A$ as
\begin{equation}\label{eq:Aal2}
A=\frac{1}{\Dt}\int_0^{\Dt}\int_{\Dom^2} \frac{(|\wdi-\wde|^2-|\wdi-\wde(t,y)|^2)}{a_{\Dx}}\mathbf{1}_{[\nu,\tau)}(t)\omega_{\epsx}(\xi-\Dt)\partial_{ss}\omega_{\epst}(t-s+\xi)\, \duz\, d\xi,
\end{equation}
and bound in the following way, using triangle and Cauchy-Schwarz inequality,
\begin{align*}
\begin{split}
|A|&\leq \frac{1}{\Dt}\int_0^{\Dt}\int_{\Dom^2} \frac{|\wde-\wde(t,y)| |2\wdi-\wde-\wde(t,y)|}{a_{\Dx}}\mathbf{1}_{[\nu,\tau)}(t)\omega_{\epsx}|\xi-\Dt| \big|\partial_{ss}\omega_{\epst}\big|\, \duz\, d\xi,\\
&\leq \frac{1}{\Dt}\int_0^{\Dt}\!\!\int_{\nu}^\tau\!\!\int_0^T |\xi-\Dt| \left(\int_{\xdom^2} \frac{|\wde-\wde(t,y)|^2}{a_{\Dx}}\omega_{\epsx}\, dxdy\right)^{\frac{1}{2}}\\
&\hphantom{\leq \frac{1}{\Dt}\int_0^{\Dt}\!\!\int_{\nu}^\tau\!\!\int_0^T}\times \biggl[\left(\int_{\xdom^2} \frac{|\wdi-\wde(t,y)|^2}{a_{\Dx}}\omega_{\epsx}\, dxdy\right)^{\frac{1}{2}}+\left(\int_{\xdom^2} \frac{|\wdi-\wde|^2}{a_{\Dx}}\omega_{\epsx}\, dxdy\right)^{\frac{1}{2}}\biggr]\big|\partial_{ss}\omega_{\epst}\big|\,dsdtd\xi\\
&\leq C\frac{2 \Dt}{\epst^2}\int_{\nu}^\tau\!\! \sup_{s\in[t-2\epst,t+2\epst]}\left(\int_{\xdom^2} \frac{|\wde-\wde(t,y)|^2}{a_{\Dx}}\omega_{\epsx}\, dxdy\right)^{\frac{1}{2}}\!\! \sup_{s\in[t-2\epst,t+2\epst]}\left(\int_{\xdom^2} \frac{|\wdi-\wde|^2}{a_{\Dx}}\omega_{\epsx}\, dxdy\right)^{\frac{1}{2}}\!\!  dt\\
&\leq C\frac{\Dt }{\sqrt{\ua}\epst^{2-\gamma_\infty}}\int_{\nu}^\tau  \sup_{s\in[t-2\epst,t+2\epst]}\left(\int_{\xdom^2} \frac{|\wdi-\wde|^2}{a_{\Dx}}\omega_{\epsx}\, dxdy\right)^{\frac{1}{2}}  dt
\end{split}
\end{align*}
We denote
\begin{equation*}
\kappa(t)=\int_0^T\int_{\xdom^2}\frac{|\wde(s-\Dt,y)-\wdi(t,x)|^2}{a_{\Delta}(x)}\omega_{\epsx}\omega_{\epst}\, dxdyds
\end{equation*}
and observe that
\begin{align}\label{eq:kappa1}
\begin{split}
&\int_{\nu}^\tau\sup_{s\in[t-2\epst,t+2\epst]}\left(\int_{\xdom^2} \frac{|\wdi(t,x)-\wde(s,y)|^2}{a_{\Dx}(x)}\omega_{\epsx}\, dxdy\right)^{1/2}dt\\
&\quad \leq  \int_{\nu}^\tau\biggl\{\sup_{s\in[t-2\epst,t+2\epst]}\left(\int_{\xdom^2} \frac{|\wde(t,y)-\wde(s,y)|^2}{a_{\Dx}}\omega_{\epsx}\, dxdy\right)^{1/2}\\
&\quad \hphantom{\leq  \int_{\nu}^\tau\biggl\{}  +\left(\int_{\xdom^2} \frac{|\wdi-\wde(t,y)|^2}{a_{\Dx}}\omega_{\epsx}\, dxdy\right)^{1/2}\biggr\}dt\\
&\quad\leq C T\epst^{\gamma_\infty} +\int_{\nu}^\tau \left(\int_{\xdom^2} \frac{|\wdi(t,x)-\wde(t,y)|^2}{a_{\Dx}}\omega_{\epsx}\, dxdy\right)^{1/2}dt\\
&\quad\leq C T\epst^{\gamma_\infty} +\int_{\nu}^\tau \biggl\{ \left(\int_0^T\int_{\xdom^2} \frac{|\wdi(t,x)-\wde(s-\Dt,y)|^2}{a_{\Dx}}\omega_{\epsx}\omega_{\epst}\, dxdyds\right)^{1/2}\\
&\quad \hphantom{\leq C T\epst^{\gamma_\infty} +\int_{\nu}^\tau\biggl\{}+ \left(\int_0^T\int_{\xdom^2} \frac{|\wde(t,y)-\wde(s-\Dt,y)|^2}{a_{\Dx}}\omega_{\epsx}\omega_{\epst}\, dxdyds\right)^{1/2}\biggr\}dt\\
&\quad \leq C T\epst^{\gamma_\infty} +\int_{\nu}^\tau \sqrt{\kappa(t)}\,dt.
\end{split}
\end{align}
Therefore, the term $A$ can be bounded as
\begin{equation}\label{eq:Al2}
|A|\leq C\frac{\Dt }{\epst^{2-\gamma_\infty}}\biggl\{\epst^{\gamma_\infty} +\int_{\nu}^\tau \sqrt{\kappa(t)}\,dt \biggr\}.
\end{equation}
We proceed to estimating the term $B$:
\begin{equation*}
B:=\int_{\Dom^2} \frac{|\wdi-\wde|^2}{a_{\Dx}\ad}\mathbf{1}_{[\nu,\tau)}(t)\omega_{\epsx} \partial_s \omega_{\epst}\left(a_{\Dx}-\ad\right)\, \duz
\end{equation*}
We again use that
\begin{equation*}
\int_{\Dom^2} \frac{|\wdi-\wde(t,y)|^2}{a_{\Dx}\ad}\mathbf{1}_{[\nu,\tau)}(t)\omega_{\epsx} \partial_s \omega_{\epst}\left(a_{\Dx}-\ad\right)\, \duz=0,
\end{equation*}
which admits us to rewrite the term $B$ and estimate as follows:
\begin{align}\label{eq:Bl2}
\begin{split}
|B|&\leq \int_{\Dom^2} \frac{|\wde-\wde(t,y)| |2\wdi-\wde- \wde(t,\cdot)|}{a_{\Dx}\ad}\mathbf{1}_{[\nu,\tau)}(t)\omega_{\epsx} |\partial_s \omega_{\epst}|\big|a_{\Dx}-\ad\big|\, \duz\\
&\leq \frac{1}{\ua} \int_{\nu}^\tau\!\!\int_0^T\!\!\int_{\xdom^2}\frac{|\wde-\wde(t,y)| |2\wdi-\wde- \wde(t,\cdot)|}{a_{\Dx}}\omega_{\epsx} |\partial_s \omega_{\epst}|\left(\big|a_{\Dx}-\ad(x)\big|+\big|\ad(x)-\ad\big|\right)\, \duz\\
&\leq \frac{1}{\ua}\left(\|a_{\Dx}-\ad\|_{\infty}+ \sup_{|x-y|\leq 2\epsx}|\ad(x)-\ad(y)|\right)\\
&\qquad\qquad\qquad\times  \int_{\nu}^\tau\!\!\int_0^T\!\!\int_{\xdom^2} \frac{|\wde-\wde(t,y)| |2\wdi-\wde- \wde(t,\cdot)|}{a_{\Dx}} \omega_{\epsx} |\partial_s \omega_{\epst}|\, \duz\\
&\leq C\frac{\Dx^{\alpha}+\delta^{\alpha}+\epsx^{\alpha}}{\ua}\int_{\nu}^\tau\!\!\int_0^T\!\!\int_{\xdom^2} \frac{|\wde-\wde(t,y)| |2\wdi-\wde- \wde(t,\cdot)|}{a_{\Dx}} \omega_{\epsx} |\partial_s \omega_{\epst}|\, \duz\\
&\leq C\frac{\Dx^{\alpha}+\delta^{\alpha}+\epsx^{\alpha}}{\ua^{3/2}}\int_{\nu}^\tau\!\!\sup_{s\in[t-\epst,t+\epst]}\left(\int_{\xdom^2}|\wde-\wde(t,y)|^2\omega_{\epsx}\,dxdy\right)^{1/2}\\
&\hphantom{\leq C\frac{\Dx^{\alpha}+\delta^{\alpha}+\epsx^{\alpha}}{\ua^{3/2}}\int_{\nu}^\tau} \times\sup_{s\in[t-\epst,t+\epst]}\left( \int_{\xdom^2} \frac{|\wdi-\wde|^2}{a_{\Dx}} \omega_{\epsx}dxdy\right)^{1/2}\int_0^T\!|\partial_s \omega_{\epst}|\, dsdt\\
&\leq C \frac{\Dx^{\alpha}+\delta^{\alpha}+\epsx^{\alpha}}{\epst^{1-\gamma_\infty}}\int_{\nu}^{\tau}\sup_{s\in[t-\epst,t+\epst]}\left( \int_{\xdom^2} \frac{|\wdi-\wde|^2}{a_{\Dx}} \omega_{\epsx}dxdy\right)^{1/2} dt\\
&\leq C \frac{\Dx^{\alpha}+\delta^{\alpha}+\epsx^{\alpha}}{\epst^{1-\gamma_\infty}}\left( \epst^{\gamma_\infty} +\int_{\nu}^{\tau}\sqrt{\kappa(t)} dt\right)
\end{split}
\end{align}
where we have used \eqref{eq:kappa1} for the last inequality. We continue to estimate the term $D$. We have, using \eqref{eq:weissnoed}
\begin{align*}
D&=\int_{\Dom^2} |\wdi-\wde|^2\mathbf{1}_{[\nu,\tau)}(t)\omega_{\epst}\left(\partial_y\omega_{\epsx}+ D^+_x\omega_{\epsx}\right)\, \duz \\
&=\frac{1}{\Dx}\int_0^{\Dx}\!\!\int_{\Dom^2} |\wdi-\wde|^2\mathbf{1}_{[\nu,\tau)}(t)\omega_{\epst}(\Dx-\xi) \partial_{yy}\omega_{\epsx}(x-y+\xi)\, \duz\, d\xi \\
&=\frac{1}{\Dx}\int_0^{\Dx}\!\!\int_{\Dom^2}\left( |\wdi-\wde|^2- |\wdi-\wde(s,x)|^2\right)\mathbf{1}_{[\nu,\tau)}(t)\omega_{\epst}(\Dx-\xi) \partial_{yy}\omega_{\epsx}(x-y+\xi)\, \duz\, d\xi.
\end{align*}
Hence we can estimate the term $D$ by
\begin{align*}
|D|&\leq \frac{1}{\Dx}\int_0^{\Dx}\!\!\int_{\nu}^{\tau}\!\!\int_{\Dom}\int_{\xdom}|\wde- \wde(\cdot,x)| |2 \wdi-\wde-\wde(\cdot,x)| \omega_{\epst}|\Dx-\xi| \big|\partial_{yy}\omega_{\epsx}\big|\, \duz\, d\xi\\
&\leq 2\int_0^{\Dx}\!\!\int_{\nu}^{\tau}\!\!\sup_{|h|\leq 2\epsx}\left(\int_{\Dom}|\wde(\cdot,x+h)- \wde(\cdot,x)|^2\omega_{\epst}dxds\right)^{\frac{1}{2}}\\
&\hphantom{ 2\int_0^{\Dx}\!\!\int_{\nu}^{\tau}}\, \times\sup_{|h|\leq 2\epsx}\left(\int_{\Dom}|\wdi- \wde(\cdot,x+h)|^2\omega_{\epst}dxds\right)^{\frac{1}{2}}\int_{\xdom}\!\ \big|\partial_{yy}\omega_{\epsx}\big|\, dydt  d\xi\\
&\leq C\frac{\Dx}{\epsx^{2-\gamma_\infty}}\int_{\nu}^{\tau}\sup_{|h|\leq 2\epsx}\left(\int_{\Dom}|\wdi- \wde(\cdot,x+h)|^2\omega_{\epst}dxds\right)^{\frac{1}{2}}dt
\end{align*}
We have
\begin{align}\label{eq:kappa2}
\begin{split}
&\int_{\nu}^\tau\sup_{|h|\leq 2\epsx}\left(\int_{\Dom}|\wdi(t,x)-\wde(s,x+h)|^2\omega_{\epst}\, dxds\right)^{1/2}dt\\
&\quad \leq \int_{\nu}^\tau\biggl\{\sup_{|h|\leq 2\epsx}\left(\int_{\Dom}|\wde(s,x)-\wde(s,x+h)|^2\omega_{\epst}\, dxds\right)^{1/2}\\
&\quad \hphantom{\leq  \int_{\nu}^\tau\biggl\{}+\sup_{|h|\leq 2\epsx}\left(\int_{\Dom}|\wdi(t,x)-\wde(s,x)|^2\omega_{\epst}\, dxds\right)^{1/2}\biggr\}dt\\
&\quad\leq C\epsx^{\gamma_\infty} +\int_{\nu}^\tau \left(\int_{\Dom}|\wdi(t,x)-\wde(s,x)|^2\omega_{\epst}\omega_{\epsx}\, dxdyds\right)^{1/2}dt\\
&\quad\leq C\epsx^{\gamma_\infty} +\int_{\nu}^\tau\biggl\{ \left(\int_{\xdom}\int_{\Dom}|\wde(s,x)-\wde(s-\Dt,y)|^2\omega_{\epst}\omega_{\epsx}\, dxdyds\right)^{1/2}\\
&\quad \hphantom{\leq C\epsx^{\gamma_\infty} + \int_{\nu}^\tau\biggl\{} +\left(\int_{\xdom}\int_{\Dom}|\wdi(t,x)-\wde(s-\Dt,y)|^2\omega_{\epst}\omega_{\epsx}\, dxdyds\right)^{1/2}\biggr\}dt\\
&\quad\leq C (\epsx^{\gamma_\infty}+\epst^{\gamma_\infty}) + C \oa\int_{\nu}^\tau \sqrt{\kappa(t)}\, dt,
\end{split}
\end{align}
and consequently,
\begin{equation}\label{eq:Dl2}
|D|\leq C \frac{\Dx}{\epsx^{2-\gamma_\infty}}\left(\epsx^{\gamma_\infty}+\epst^{\gamma_\infty} +\int_{\nu}^\tau \sqrt{\kappa(t)}\, dt\right).
\end{equation}
Summing up, equation \eqref{eq:schissix} becomes
\begin{equation}\label{eq:schissixy}
\frac{1}{\Dt}\int_{\tau}^{\tau+\Dt} \kappa(t)\, dt
\leq \frac{1}{\Dt}\int_{\nu}^{\nu+\Dt} \kappa(t)\,dt+M_1+M_2 \int_{\nu}^{\tau} \sqrt{\kappa(t)}\, dt
\end{equation}
where
\begin{equation*}
M_1=C\left(\frac{\Dt }{\epst^{2-2\gamma_\infty}}+\frac{\delta^{\alpha}+\epsx^{\alpha}}{\epst^{1-2\gamma_\infty}} + \frac{\Dx}{\epsx^{2-2\gamma_\infty}}+ \frac{\Dx\,\epst^{\gamma_\infty}}{\epsx^{2-\gamma_\infty}} \right), \quad M_2= C\left(\frac{\Dt }{\epst^{2-\gamma_\infty}}+ \frac{\delta^{\alpha}+\epsx^{\alpha}}{\epst^{1-\gamma_\infty}} +\frac{\Dx}{\epsx^{2-\gamma_\infty}} \right)
\end{equation*}
We choose $\nu$ and $\tau$ such that $\tau/\Dt,\nu/\Dt\in \N$, i.e. $\nu=N_1\Dt$ and $\tau=N_2\Dt$ for some $N_1,N_2\in\N$ and notice that
\begin{equation*}
\int_{k\Dt}^{(k+1)\Dt} \sqrt{\kappa(t)} \,dt \leq \Dt\, \sqrt{\frac{1}{\Dt}\int_{k\Dt}^{(k+1)\Dt}\!\!\! \kappa(t)\, dt}:=\Dt\, X_k.
\end{equation*}
Hence we can rewrite equation \eqref{eq:schissixy} as 
\begin{equation*}%\label{eq:whatever}
X_{N_2}^2\leq X_{N_1}^2 + M_1+\Dt M_2 \sum_{i={N_1}}^{N_2-1} X_i.
\end{equation*}
Now we use the following simple adaption of \cite[[Theorem 5, page 4]{gronwall1}:
\begin{lemma}\label{lem:discretegronwall}
	Let $X_0\in\R_{\geq 0}$,  $X_k\geq 0$, $k=1,\dots, N$ for some $N\in \mathbb{N}$ satisfy
	\begin{equation}\label{eq:lemmacond}
	X_k^2\leq X_0^2 + C_1 + C_2 \sum_{i=0} ^k X_i,
	\end{equation}  
	for all $k\in\{1,\dots, N\}$, for some $C_1, C_2\geq 0$. Then
	\begin{equation*}
	X_k\leq X_0+ \sqrt{C_1} + C_2 k.
	\end{equation*}
\end{lemma}
Using this lemma with $C_1=M_1$ and $C_2 = \Dt M_2$, we obtain the estimate
\begin{equation}\label{eq:jumalauta}
X_{N_2}\leq X_{N_1} + \sqrt{M_1} + \Dt M_2 (N_2-N_1)= X_{N_1}+ \sqrt{M_1} + T M_2
\end{equation}

Next, we relate the $L^2$-norm of the difference $(\wde-\wdi)(\tau)$ to $X_{N_2}$. Indeed,
\begin{align*}
&\Big| \|(\wde-\wdi)/\sqrt{a_{\Dx}}\|_{L^2}(\tau)- X_{N_2}\Big|\\
&\quad \leq \left(\frac{1}{\Dt}\int_{\tau}^{\tau+\Dt}\int_{\Dom}\!\int_{\xdom}\frac{|\wde(\tau,x)-\wde(s-\Dt,y)|^2}{a_{\Dx}}\omega_{\epsx}\omega_{\epst}\duz\right)^{1/2}\\
&\quad \leq C (\epst^{\gamma_\infty}+\epsx^{\gamma_\infty})
\end{align*}
In a similar way, we can show
\begin{equation*}
\Big| \|(\wde_0-\wdi(0,\cdot))/\sqrt{a_{\Dx}}\|_{L^2}- X_{N_1}\Big|\leq C(\epst^{\gamma_\infty}+\epsx^{\gamma_\infty}+\nu^{\gamma_\infty}),
\end{equation*}
and therefore, with \eqref{eq:jumalauta},
\begin{equation*}
\|\wde-\wdi\|_{L^2}(\tau)\leq C \left( \|\wde_0-\wdi(0,\cdot)\|_{L^2}+ \epsx^{\gamma_\infty}+\epst^{\gamma_\infty} +\nu^{\gamma_\infty}+\sqrt{M_1} + T M_2\right)
\end{equation*}
Letting $\delta\rightarrow 0$ and inserting the definitions of $M_1$ and $M_2$, this is

\begin{multline*}
\|w-\wdi\|_{L^2}(\tau)\leq C \biggl( \|w_0-\wdi(0,\cdot)\|_{L^2}+ \epsx^{\gamma_\infty}+\epst^{\gamma_\infty} +\nu^{\gamma_\infty}+ \frac{\Dt^{1/2} }{\epst^{1-\gamma_\infty}}\\
+\frac{\epsx^{\alpha/2}}{\epst^{1/2-\gamma_\infty}} + \frac{\Dx^{1/2}}{\epsx^{1-\gamma_\infty}}+ \frac{\Dx^{1/2}\,\epst^{\gamma_\infty/2}}{\epsx^{1-\gamma_\infty/2}}
+\frac{\Dt }{\epst^{2-\gamma_\infty}}+ \frac{\epsx^{\alpha}}{\epst^{1-\gamma_\infty}} +\frac{\Dx}{\epsx^{2-\gamma_\infty}}\biggr)
\end{multline*}
Now we can choose $\nu=3\epst$, $\epst=\epsx^{\alpha}$ and $\epsx= \Dx^{1/(\alpha\gamma_\infty+2-\gamma_\infty)}$ to balance the errors and finally obtain \eqref{eq:rae1dl2}.

\end{proof}
\begin{proof}[Proof f Lemma \ref{lem:discretegronwall}]
Define $Y_k:= X_0^2 + C_1 + C_2 \sum_{i=0}^k X_i$. Then by \eqref{eq:lemmacond}, $X_k^2\leq Y_k$. Moreover, subtracting the expression for $Y_{k-1}$ from the expression for $Y_k$, we have
\begin{equation*}
Y_k-Y_{k-1}= C_2 X_k\leq C_2 \sqrt{Y_k}\leq C_2 \left(\sqrt{Y_k} + \sqrt{Y_{k-1}}\right)
\end{equation*}	
Since $Y_k-Y_{k-1}=(\sqrt{Y_k}-\sqrt{Y_{k-1}})(\sqrt{Y_k}+\sqrt{Y_{k-1}})$, we can divide both sides of the above equation by $(\sqrt{Y_k}+\sqrt{Y_{k-1}}$ to obtain
\begin{equation*}
\sqrt{Y_k}-\sqrt{Y_{k-1}}\leq C_2.
\end{equation*}	
Using induction over $k$, we obtain
\begin{equation*}
\sqrt{Y_k}\leq \sqrt{Y_0}+ C_2 k.
\end{equation*}
Hence
\begin{equation*}
X_k\leq \sqrt{Y_k}\leq \sqrt{Y_0}+ C_2 k= \sqrt{X_0+ C_1} + C_2 k,
\end{equation*}
by the definition of $Y_k$. Using that $\sqrt{a^2+b^2}\leq |a|+ |b|$, this proves the claim.
\end{proof}

\subsection{Experimental rates for the advection equation}\label{ssec:exprates}
In this section, we run a few numerical experiments to compare the theoretically established rates to experimentally observed ones. As a model coefficient $a$, we choose a sample (single realization) of a log-normally distributed random field, which was
 generated using a spectral FFT method
 \cite{cd1997,pico1993,rnh2000,Muell455} from a given covariance
 operator $\hat{c}$ which we assume to be log-normal, so that the
 covariance operator completely determines the law of $\hat{c}$. It is
 easy to check that this coefficient $a$ is uniformly positive,
 bounded from above and H\"{o}lder continuous with exponent $1/2$. See
 Figure \ref{fig:coef1d} for an illustration of the coefficient.
 \begin{figure}[ht] % \caption{A gull}
 	\centering
 	\includegraphics[width=0.8\textwidth]{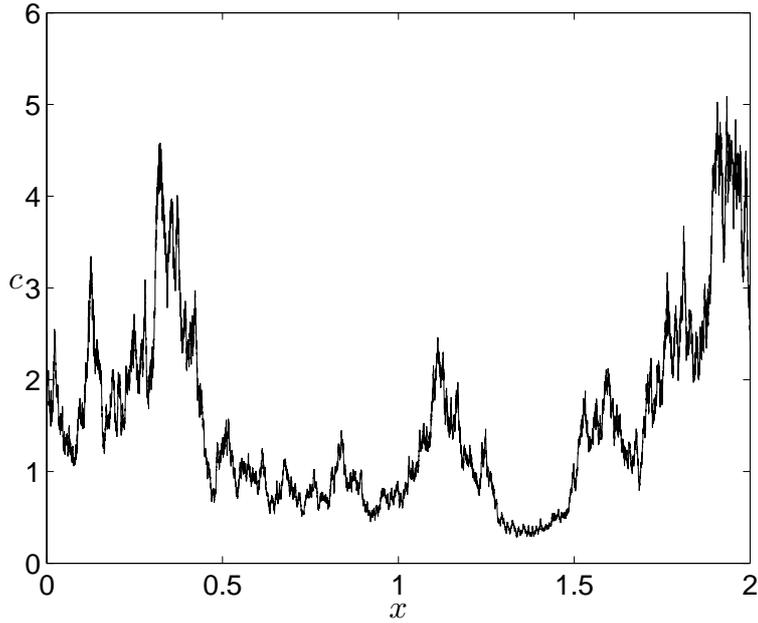}
 	\caption{The coefficient $a$ used for the numerical experiments for
 		the transport equation \eqref{eq:forw}.}
 	\label{fig:coef1d}
 \end{figure}
For the function $w_0$, we choose the product of the Lipschitz continuous hat function
\begin{equation*}
h(x)=\begin{cases}
	1+2(x-\nicefrac{1}{2}),&\quad x\in[0,0.5)\\
	1-2(x-\nicefrac{1}{2})&\quad x\in[0.5,1),\\
	1+2(x-\nicefrac{3}{2}),&\quad x\in[1,1.5),\\
	1-2(x-\nicefrac{3}{2}),&\quad x\in[1.5,2),
\end{cases}
\end{equation*}
with a modification of the Weierstrass function for different parameters $\gamma$:
\begin{equation}\label{eq:weier}
f^\gamma(x)=\sum_{n=1}^\infty 2^{-\gamma n}\cos(2^n \pi x),\quad\gamma\in (0,1),
\end{equation}
that is, we add a constant $f_0>0$ such that $f^\gamma$ becomes strictly
positive, and truncate $f^\gamma$ after $N=400$ terms,
\begin{equation}\label{eq:acoef}
\widetilde{f}^\gamma(x)=\sum_{n=1}^{400} 2^{-\gamma n}\cos(2^n \pi x)+f_0,\quad \alpha\in (0,1).
\end{equation}
and define $w_{0,\gamma}(x)=h(x)\widetilde{f}^\gamma(x)$.

It can be shown that \eqref{eq:weier} is nowhere differentiable but
H\"{o}lder continuous with exponent $\gamma$. As a computational
domain, we take $D=[0,2]$ with periodic boundary conditions.
We run experiments up to time $T=1$ with CFL-number $\theta=\nicefrac{0.4}{\overline{a}}$ with initial data $w_{0,\gamma}$ for $\gamma=\nicefrac{1}{2},\nicefrac{1}{4},\nicefrac{1}{8}$ and for $w_{0,0}(x):=h(x)$.

To approximate the coefficient, we interpolate \eqref{eq:acoef} and $a$ on a grid with meshsize $\Dx=2^{-14}$ and average it to obtain an approximation on the coarser grids. The reference solution has been computed on a grid with
$N_x=2^{14}$ mesh points. We have used the following approximation for
the numerical convergence rate
\begin{equation}\label{eq:rateapprox}
r^m=\frac{1}{N_{\mathrm{exp}}-1}\sum_{k=1}^{N_{\mathrm{exp}}-1}\frac{\log{\Epsilon^m_{\Dx_k}}-\log{\Epsilon^m_{\Dx_{k-1}}}}{\log{2}},\quad m=1,2
\end{equation}
where $\Dx_k=2^{-k}\Dx_0$ and $\Epsilon^m_{\Dx_k}$, the relative distance
of the approximation with gridsize $\Dx_k$ to the reference solution
in the discrete $L^m$-norm, that is,
\begin{align}\label{eq:relerr}
\Epsilon_{\Dx_k}^m&=100\times
\frac{\sum_{j=1}^{N_x}|u_{\Dx_k}(T,x_j)-u_{\Dx_{\mathrm{ref}}}(T,x_j)|^m}{\sum_{j=1}^{N_x}|u_{\Dx_{\mathrm{ref}}}(T,x_j)|^m}.
\end{align}

We used $\Dx_0=1/16$ ($N_{x,0}=32$) and $N_{\mathrm{exp}}=6$. In Figure
\ref{fig:transport_ref}, we have plotted the Weierstrass function and
the reference solution for H\"{o}lder exponent
$\gamma=1/2$. Interestingly, the variable $w$ seems to be much
smoother at time $T=1$ than initially and also much smoother than the
variable $u$. This is probably due to the diffusion in the scheme. In Table \ref{tab:er} the experimentally observed rates are computed for initial data $w_{0,\gamma}$ and $w_{0,0}$ for $\gamma=2^{-k}$, $k=2,4,8$. We notice that the experimental rates for this example are low but better than what we obtain from the theoretical estimates. This can be due to the fact that we compute the errors with respect to a reference solution computed by the same scheme. Moreover, other examples of initial data might give lower rates. However, we do not know whether the rates \eqref{eq:rae1d} and \eqref{eq:rae1dl2} are sharp.
\begin{figure}[ht] % \caption{A gull}
	\centering
	\begin{tabular}{lr}
		\hspace{-2em}\includegraphics[width=0.56\textwidth]{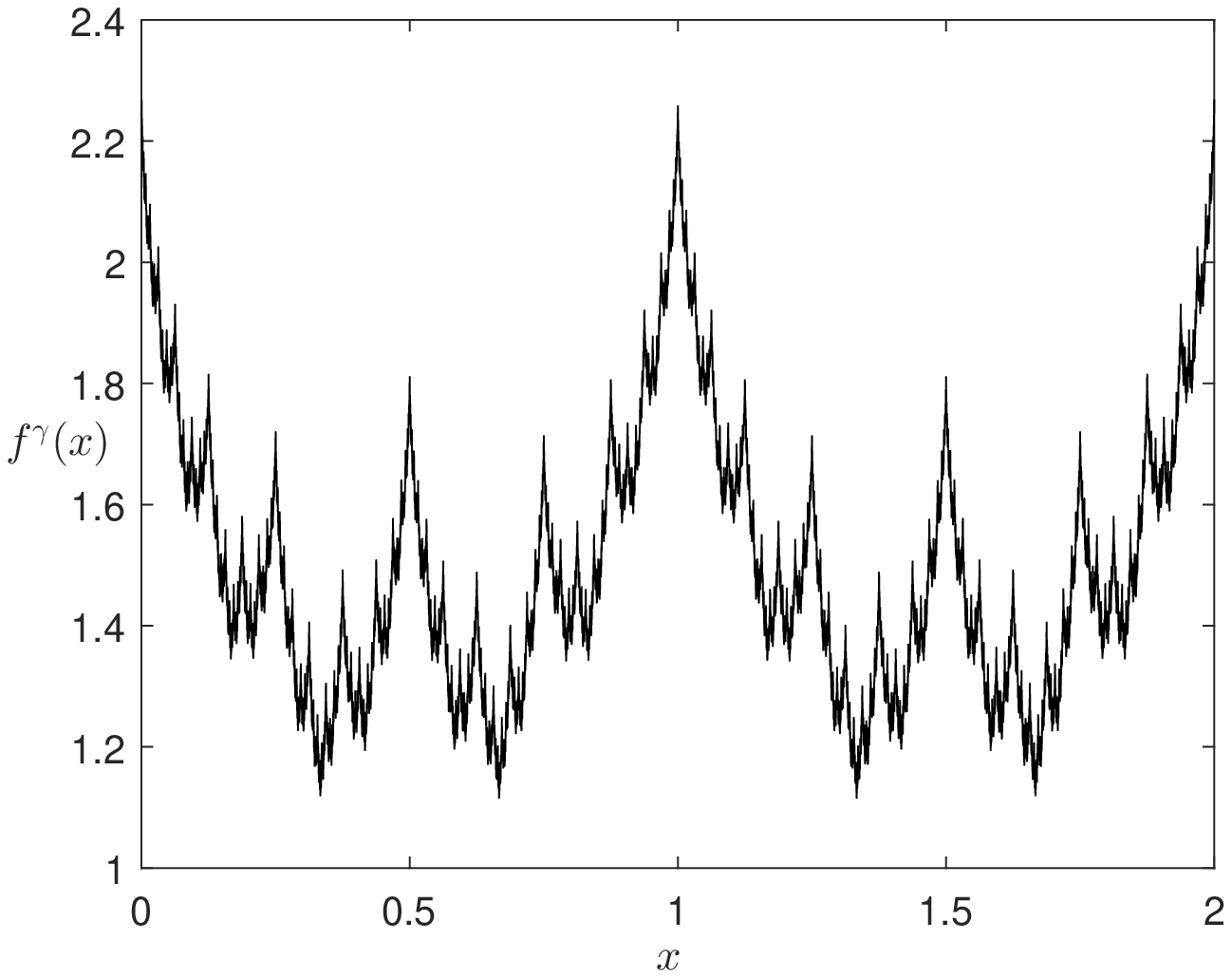}\hspace{-2em}
		\includegraphics[width=0.56\textwidth]{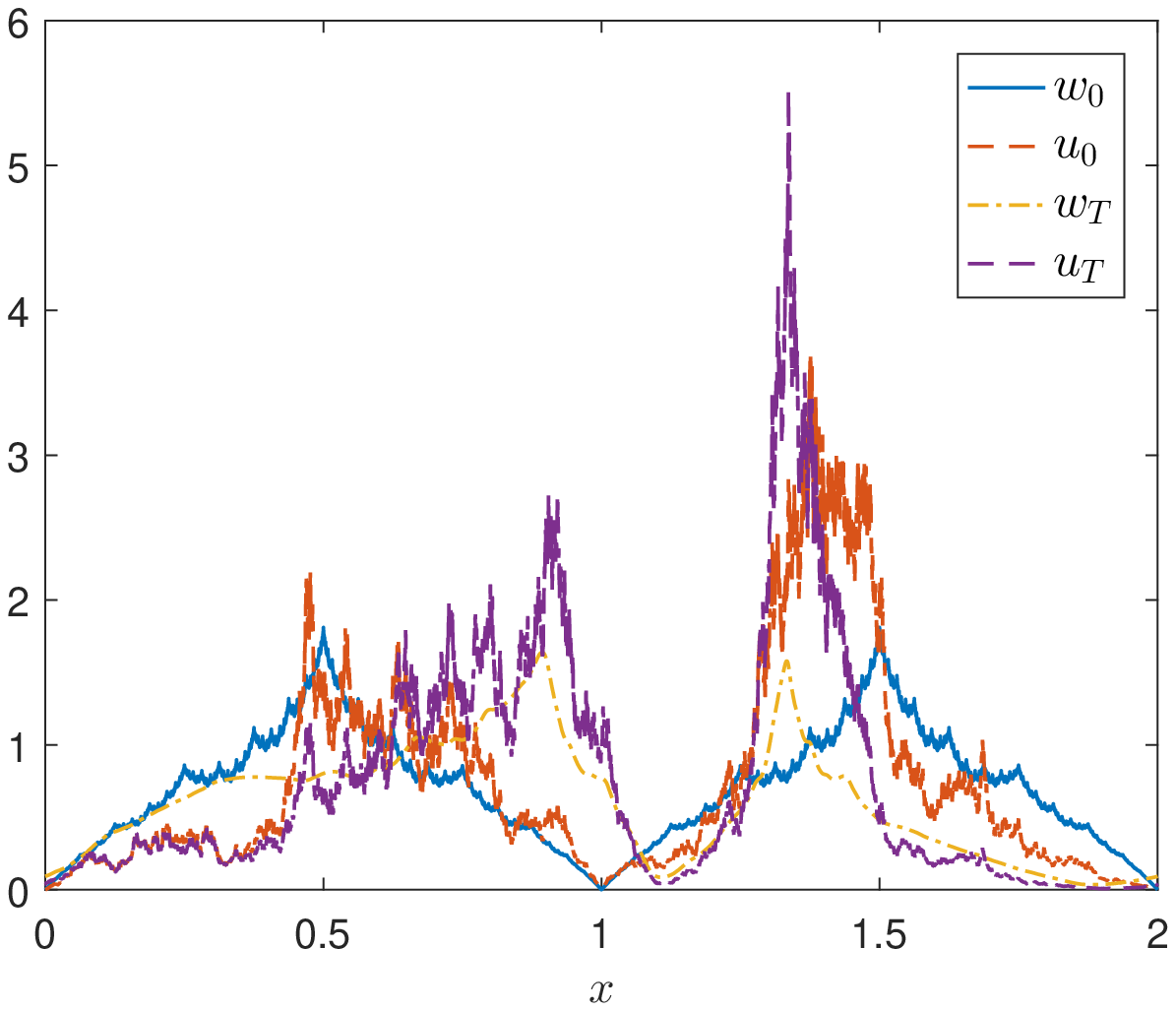}
	\end{tabular}
	\caption{Left: Approximation of Weierstrass function
		\eqref{eq:acoef} for $\gamma=1/2$. Right: Approximation of
		\eqref{eq:forw} by scheme \eqref{eq:upwindscheme}
		at time $T=0$ and $T=1$, $N_x=2^{14}$, $\gamma=1/2$.}
	\label{fig:transport_ref}
\end{figure}
\begin{table}[h]
	\centering
	\begin{tabular}[h]{|c|c|c|c|c|}
		\hline
	$\gamma$  &$ r^1_u$ & $r^1_w$ & $r^2_u$ & $r^2_w$\\
		\hline
		 $1$ &0.6018 &   0.5598 &   0.6468 &   0.5829 \\
		$1/2$ & 0.5170  &  0.4554  &  0.5400  &  0.4996\\
		$1/4$ &  0.4412 &   0.3816 &   0.4678 &   0.4356\\
		$1/8$ &  0.4550 &   0.3970 &   0.4810 &   0.4484\\
		\hline
	\end{tabular}
	\vspace{1mm}
	\caption{Experimental rates}
	\label{tab:er}
\end{table}

\section{A convergence rate for the wave equation in one space dimension}
\label{sec:1d}
The techniques from the last section can be used to prove a rate of convergence for approximate solutions to the acoustinc wave equation in one space dimension with rough coefficient under some assumptions. Defining $u:=p_x$ and $v:= p_t $, the second order wave equation
 \begin{equation*}
 \frac{1}{a(x)} p_{tt}(t,x)- p_{xx}(t,x)=0,\quad (t,x)\in D_T,
 \end{equation*}
 $D_T:=[0,T]\times D$, where $D=[d_L,d_R]$, $d_L<d_R\in [-\infty,\infty]$, can be rewritten as
\begin{align}\label{eq:1dwavew}
\begin{split}
u_t(t,x)-v_x(t,x)&=0,\\
\frac{1}{a(x)}v_t(t,x)-u_x(t,x)&=0, \quad (t,{x})\in D_T,
\end{split}
\end{align}
For simplicity, let us assume that $D=[0,2]$ with periodic boundary conditions.

\subsection{Numerical approximation of \eqref{eq:1dwavew} by a finite
  difference scheme}\label{ssec:rsys1}
In order to compute numerical approximations to \eqref{eq:1dwavew}, we
choose $\Delta x>0$ and discretize the spatial domain by a grid with
gridpoints $x_{j+1/2}:=j \Delta x$, $j\in \{1,\dots, N_x\}$, where $N_x\in\N$ is such that $N_x \Dx=|D|$. Similarly let
$\Dt$ denote the time step and $t^n = n \Dt$ with $n=0,1,\dots, N$
denote the $n$-th time level with $N\Dt = T$.

We define the averaged quantities

\begin{equation}\label{eq:cav}
  a_j=\frac{1}{\Delta x}\int_{x_{j-1/2}}^{x_{j+1/2}} a(x)\, dx,\quad j=1,\dots, N_x,
\end{equation}
and

\begin{equation}\label{eq:initdataapproxwave}
  \left(u_j^0,v_j^0\right)=\frac{1}{\Delta x}\left(\int_{x_{j-1/2}}^{x_{j+1/2}} u_0(x)\, dx, \int_{x_{j-1/2}}^{x_{j+1/2}} v_0(x) dx \right) \quad j=1,\dots, N_x.
\end{equation}

We recall \eqref{eq:Dt} and define approximations to \eqref{eq:1dwavew} by the finite
difference scheme:
\begin{subequations}\label{eq:numwave1}
  \begin{align}\label{eq:num11}
    D_t^+ u_j^n&=D_x^c v_j^n+\frac{\Dx}{2} D_x^+ D_x^- u_j^n,\\
    \label{eq:num12}
    \frac{D_t^+ v_j^n}{a_j}& = D_x^c u_j^n+ \frac{\Dx}{2} D_x^+ D_x^-
    v_j^n, \quad j\in \Z, n=1,\dots, N,
  \end{align}
\end{subequations}
with the time step $\Dt$ being chosen such that the CFL-condition,
\begin{equation}\label{eq:cflc}
  2\Dt \max_j\seq{\max\seq{2a_j+1,a_j/4+5/4}} \le \Dx
\end{equation}
is satisfied.

Moreover for any $k,l \in \R$, we define the discrete entropy (energy)
function and flux
\begin{equation*}
  \eta_j^n:=\frac{|u_j^n-k|^2}{2}+\frac{|v_j^n-\ell|^2}{2 a_j},\quad q_j^n:=-(u_j^n-k)(v_j^n-\ell).
\end{equation*}
Furthermore, we will for technical reasons need the following
difference quotients: We denote for $\gamma\in (0,1]$ and a discrete
quantity $\sigma_j^n$ defined on the grid,
\begin{equation}\label{eq:doofinotation}
  D^\pm_{\gamma,t} \sigma^n_j=\mp \frac{\sigma^n_j-\sigma_{j}^{n\pm 1}}{\Dt^\gamma},\quad D^\pm_{\gamma,x} \sigma^n_j = \mp \frac{\sigma^n_j-\sigma_{j\pm 1}^n}{\Dx^\gamma},\quad
  D^c_{\gamma,x} \sigma^n_j   = \frac{\sigma^n_{j+1}-\sigma_{j- 1}^n}{2\Dx^\gamma}.
\end{equation}
When $\gamma=1$, the above defined quantities reduce to the usual finite differences \eqref{eq:Dt}. The scheme \eqref{eq:numwave1} satisfies the following properties:

\begin{lemma}\label{lem:prowave1}
  Assume $a\in C^{0,\alpha}(D)$ and $u_0, v_0\in L^2(D)$. Then the
  numerical approximations $u_j^n$ and $v_j^n$ defined by
  \eqref{eq:numwave1}, \eqref{eq:cav} and
  \eqref{eq:initdataapproxwave} have the following properties:
  \begin{enumerate}
  \item[(i)] Discrete entropy inequality:
    \begin{multline}\label{eq:entrw11}
      \Dtp \eta^n_j + \Dc q^n_j \le \frac{\Dx\left(\Dt-\Dx\right)}{2}
      \Dm\left(\Dp
        \left(u^n_j-k\right)\Dp \left(v^n_j-l\right)\right) \\
      \begin{aligned}
        +\frac{\Dx}{4} \Dp\Dm\left(\left(u^n_j-k\right)^2 +
          \left(v^n_j-l\right)^2\right).
      \end{aligned}
    \end{multline}
  \item[(ii)] Bounds on the discrete $L^2$-norms:
    \begin{equation}\label{eq:wl2}
      \Dx \sum_j (u^n_j)^2+\frac{1}{a_j}(v^n_j)^2\le \Dx \sum_j
      (u^0_j)^2+\frac{1}{a_j}(v^0_j)^2\le
      \norm{u_0}_{L^2}^2+\norm{a^{-1/2}v_0}_{L^2}^2 
    \end{equation}
  \item[(iii)] For any function $w = w(x)$, define the $L^2$ modulus
    of continuity in space as $\gamma$ if,
    \begin{equation}\label{eq:initmocte}
      \nu_x^2(w,\sigma):=\sup_{\delta\leq \sigma}\int_{\R}|w(x+\delta)-w(x)|^2\, dx\leq C \, \sigma^{2\gamma}.
    \end{equation}

    If we also assume that the initial data $u_0$ and $v_0$ have
    moduli of continuity in $L^2(D)$,
    \begin{equation*}
      \nu_x^2(u_0,\sigma)\le C\sigma^{2\gamma},\quad
      \nu_x^2(v_0,\sigma)\le C\, \sigma^{2\gamma}, 
    \end{equation*}
    for some $\gamma>0$, the approximations satisfy,
    \begin{align}\label{eq:mocw11}
      \begin{split}
        &\Dx \sum_j \abs{D^+_{\gamma,t} u_j^n}^2+\frac{1}{a_j}
        \abs{D^+_{\gamma,t} v_j^n}^2\le C,\\
        &\Dx\sum_j \abs{D^c_{\gamma,x} u_j^n}^2 + \abs{D^c_{\gamma,x}
          v_j^n}^2+\frac{\Dx^2}{4} (\abs{D^+_{\gamma,x} D^-_x u_j^n}^2
        + \abs{D^+_{\gamma,x} D^-_x v_j^n}^2)\le C,
      \end{split}
    \end{align}
    for all $n=0,\dots, N_T$, where $C$ is a constant depending on $a$
    and the initial data $u_0$ and $v_0$.
  \end{enumerate}
\end{lemma}
\begin{proof}
  By linearity, it is sufficient to prove \eqref{eq:entrw11} for
  $k=l=0$. We shall use the following identities
  \begin{align}
    u^n_j \Dtp u^n_j &=\frac12 \Dtp \left(u^n_j\right)^2 - \frac{\Dt}2
    \left(\Dtp u^n_j\right)^2, \label{eq:tleibnitz}\\
    u^n_j \Dp\Dm u^n_j &= \frac12 \Dp\Dm \left(u^n_j\right)^2
    \label{eq:iprod1}
    - \frac12 \left(\left(\Dm u^n_j\right)^2 + \left(\Dp
        u^n_j\right)^2\right),\\
    \Dm\left(\Dp u^n_j \Dp v^n_j\right) &= \left(\Dp\Dm u^n_j
    \right)\Dc v^n_j +\left( \Dp\Dm v^n_j \right)
    \Dc u^n_j,\label{eq:iprod2}\\
    u^n_j \Dc v^n_j + v^n_j \Dc u^n_j &= \Dc \left(u^n_j v^n_j\right)
    - \frac{\Dx^2}{2} \Dm\left(\Dp u^n_j \Dp
      v^n_j\right).\label{eq:iprod3}
  \end{align}
  Multiplying \eqref{eq:num11} by $u^n_j$ and \eqref{eq:num12} by
  $v^n_j$ we get
  \begin{align*}
    \frac12 \Dtp\left(u^n_j\right)^2 - \frac{\Dt}{2} \left(\Dtp
      u^n_j\right)^2 &= u^n_j \Dc v^n_j + \frac{\Dx}{4}
    \Dp\Dm\left(u^n_j\right)^2 \\
    &\qquad - \frac{\Dx}{4} \left(\left(\Dm u^n_j\right)^2 + \left(\Dp
        u^n_j\right)^2\right) \\
    \frac1{2a_j} \Dtp\left(v^n_j\right)^2 - \frac{\Dt}{2a_j}
    \left(\Dtp v^n_j\right)^2 &= v^n_j \Dc u^n_j + \frac{\Dx}{4}
    \Dp\Dm \left(v^n_j\right)^2 \\
    &\qquad - \frac{\Dx}{4} \left(\left(\Dm v^n_j\right)^2 + \left(\Dp
        v^n_j\right)^2\right).
  \end{align*}
  Adding these two equations
  \begin{align*}
    \Dtp \eta^n_j &= \Dc\left(u^n_j v^n_j\right) - \frac{\Dx^2}{2}
    \Dm\left(\Dp u^n_j \Dp v^n_j\right)
    \\
    &\quad +\frac{\Dx}{4} \Dp\Dm
    \left(\left(u^n_j\right)^2+\left(v^n_j\right)^2\right)
    \\
    &\quad - \frac{\Dx}{4} \left(\left(\Dm u^n_j\right)^2+\left(\Dp
        u^n_j\right)^2 +
      \left(\Dm v^n_j\right)^2+\left(\Dp v^n_j\right)^2\right)\\
    &\quad + \frac{\Dt}{2}\left[\underbrace{ \left(\Dc v^n_j +
          \frac{\Dx}{2}\Dp\Dm u^n_j\right)^2 +a_j\left(\Dc u^n_j +
          \frac{\Dx}{2}\Dm\Dp v^n_j\right)^2}_K\right].
  \end{align*}
  We can estimate $K$ as follows
  \begin{align*}
    K&\le \frac12 \left(\left(\Dm u^n_j\right)^2+\left(\Dp
        u^n_j\right)^2 +
      a_j\left(\Dm v^n_j\right)^2+a_j\left(\Dp v^n_j\right)^2 \right)\\
    &\quad + \Dx\left(\Dp\Dm u^n_j \Dc v^n_j + a_j\Dp\Dm v^n_j \Dc
      u^n_j\right) + \frac{\Dx^2}{4} \left(\left(\Dp\Dm u^n_j\right)^2
      + a_j
      \left(\Dp\Dm v^n_j\right)^2\right)\\
    &\le \left(\Dm u^n_j\right)^2+\left(\Dp u^n_j\right)^2 +
    a_j\left(\Dm v^n_j\right)^2+a_j\left(\Dp v^n_j\right)^2 \\
    &\qquad + \Dx\left(\Dp\Dm u^n_j \Dc v^n_j + a_j\Dp\Dm v^n_j \Dc
      u^n_j\right) \\
    &= \left(\Dm u^n_j\right)^2+\left(\Dp u^n_j\right)^2 +
    a_j\left(\Dm v^n_j\right)^2+a_j\left(\Dp v^n_j\right)^2 \\
    &\qquad + \Dx \Dp\left(\Dm u^n_j \Dm v^n_j\right) +
    \Dx\left(a_j-1\right)\Dp\Dm v^n_j \Dc u^n_j,\\
    &\le \left(\Dm u^n_j\right)^2+\left(\Dp u^n_j\right)^2 +
    a_j\left(\Dm v^n_j\right)^2+a_j\left(\Dp v^n_j\right)^2 \\
    &\qquad  + \Dx \Dp\left(\Dm u^n_j \Dm v^n_j\right) \\
    &\qquad + \frac12 \abs{a_j-1} \left( \left(\abs{\Dp v_j^n} +
        \abs{\Dm v^n_j}\right)^2 + \frac14 \left(\Dm u^n_j+\Dp
        u^n_j\right)^2\right)\\
    &\le \Dx \Dp\left(\Dm u^n_j \Dm v^n_j\right) + \left(1+\frac14
      \abs{a_j-1}\right)\left(\Dm u^n_j\right)^2 +
    \left(1+\frac14 \abs{a_j-1}\right)\left(\Dp u^n_j\right)^2 \\
    &\qquad +\left(a_j+\abs{a_j-1}\right)\left(\Dm v^n_j\right)^2
    +\left(a_j+\abs{a_j-1}\right)\left(\Dp v^n_j\right)^2.
  \end{align*}
  This implies that
  \begin{align*}
    \Dtp \eta^n_j + \Dc q^n_j &\le \frac{\Dx\left(\Dt-\Dx\right)}{2}
    \Dm\left(\Dp
      u^n_j\Dp v^n_j\right) \\
    &\qquad +\frac{\Dx}{4}\Dp\Dm\left(\left(u^n_j\right)^2 +
      \left(v^n_j\right)^2\right)\\
    &\qquad \qquad +\frac12\left(\left(1+\frac14
        \abs{a_j-1}\right)\Dt-\frac{\Dx}{2}\right)\left(\Dm
      u^n_j\right)^2\\
    &\qquad \qquad +\frac12\left(\left(1+\frac14
        \abs{a_j-1}\right)\Dt-\frac{\Dx}{2}\right)\left(\Dp
      u^n_j\right)^2\\
    &\qquad \qquad +\frac12\left(\left(a_j+
        \abs{a_j-1}\right)\Dt-\frac{\Dx}{2}\right)\left(\Dm
      v^n_j\right)^2\\
    &\qquad \qquad +\frac12\left(\left(a_j+
        \abs{a_j-1}\right)\Dt-\frac{\Dx}{2}\right)\left(\Dp
      v^n_j\right)^2.
  \end{align*}
  If $\Dt$ satisfies the CFL-condition \eqref{eq:cflc}, the four last
  terms above are non-positive and \eqref{eq:entrw11} follows. The
  $L^2$ bound \eqref{eq:wl2} also follows upon summing over $j$ and
  multiplying by $\Dx$.

  By the linearity of the equation, \eqref{eq:wl2} also holds for the
  difference of two approximations computed by \eqref{eq:num11} and
  \eqref{eq:num12}, thus in particular for $D^+_{\gamma,t}u_j^n$ and
  $D^+_{\gamma,t}v_j^n$. Hence, using the handy equality
  \begin{multline}\label{eq:qer}
    \sum_j \abs{\Dtp u_j^n}^2+\frac{1}{a_j^2} \abs{\Dtp v_j^n}^2
    \\
    =\sum_j \abs{\Dc u_j^n}^2 + \abs{\Dc v_j^n}^2+\frac{\Dx^2}{4}
    \left(\abs{\Dp \Dm u_j^n}^2 + \abs{\Dp\Dm v_j^n}^2\right),
  \end{multline}
  the CFL-condition \eqref{eq:cflc}, \eqref{eq:wl2} implies
  \begin{align}\label{eq:gda}
    \Dx \sum_j \left(D^+_{\gamma,t} u_j^n\right)^2&+\frac{1}{a_j}
    \left(D^+_{\gamma,t} v_j^n\right)^2\\
    &\le \Dx \sum_j \left(D^+_{\gamma,t} u_j^0\right)^2+\frac{1}{a_j}
    \left(D^+_{\gamma,t} v_j^0\right)^2\notag
    \\
    &\le \max\{1,\overline{a}\}\Dx \sum_j \left(D^+_{\gamma,t}
      u_j^0\right)^2+\frac{1}{a_j^2} \left(D^+_{\gamma,t}
      v_j^0\right)^2
    \notag\\
    &= \max\{1,\overline{a}\}\Dx \Dt^{2-2\gamma} \sum_j \left(\Dc
      u_j^0\right)^2 + \left(\Dc v_j^0\right)^2\notag\\
    &\hphantom{= \max\{1,\overline{a}\}\Dx\, \theta^{2-2\gamma}
      \sum_j} +\frac{\Dx^2}{4} \left( \Dp\Dm u_j^0\right)^2 +
    \left(\Dp\Dm v_j^0\right)^2)
    \notag\\
    &\le \max\{1,\overline{a}\}\Dx\, \theta^{2-2\gamma} \sum_j
    \left(D^c_{\gamma,x} u_j^0\right)^2 + \left(D^c_{\gamma,x}
      v_j^0\right)^2\notag\\
    &\hphantom{= \max\{1,\overline{a}\}\Dx\, \theta^{2-2\gamma}
      \sum_j} +\frac{\Dx^2}{4} \left(\left(D^+_{\gamma,x} \Dm
        u_j^0\right)^2 + \left(D^+_{\gamma,x} \Dm
        v_j^0\right)^2\right)
    \notag\\
    &\le 2 \theta^{2-2\gamma} \max\{1,\overline{a}\}\Dx\sum_j
    \left(D^+_{\gamma,x} u_j^0\right)^2 + \left(D^+_{\gamma,x}
      v_j^0\right)^2 =:C(\alpha,u_0,v_0),\notag
  \end{align}
  where we have set $\theta=\Dt/\Dx$.  Applying \eqref{eq:qer} once
  more, we also obtain the second equation in \eqref{eq:mocw11},
  \begin{multline}\label{eq:mocux}
    \Dx\sum_j \left(D^c_{\gamma,x} u_j^n\right)^2 +
    \left(D^c_{\gamma,x} v_j^n\right)^2+\frac{\Dx^2}{4}
    \left(\left(D^+_{\gamma,x} \Dm u_j^n\right)^2 +
      \left(D^+_{\gamma,x} \Dm v_j^n\right)^2\right)
    \\
    =\theta^{2\gamma-2}\Dx\sum_j \left(D_{\gamma,t}^+ u_j^n\right)^2
    +\frac{1}{a_j^2} \left(D_{\gamma,t}^+ v_j^n\right)^2\le
    C(\alpha,u_0, v_0).
  \end{multline}
\end{proof}
Defining
\begin{subequations}\label{eq:discfcn2}
  \begin{align}\label{seq1:discvfcn}
    u_{\Dx}(t,x)&=u_j^n,\quad (t,x)\in [t^n,t^{n+1})\times [x_{j-1/2},x_{j+1/2}),\\
    \label{seq1:discmfcn2}
    v_{\Dx}(t,x)&=v_j^n,\quad (t,x)\in [t^n,t^{n+1})\times [x_{j-1/2},x_{j+1/2}),\\
    \label{seq1:discmfcn3}
    a_{\Dx}(x)&=a_j,\quad x\in [x_{j-1/2},x_{j+1/2}),
  \end{align}
\end{subequations}
we have that a subsequence  of $(u_{\Dx},v_{\Dx})_{\Dx>0}$ converges weakly to a weak solution of \eqref{eq:1dwavew}. If $\gamma=1$, equation \eqref{eq:mocw11} implies that $u$ and $v$ have an $L^2$-modulus of continuity of $\gamma=1$ in space and time and therefore by Ladyzhenskaya's theorems of interpolation of finite difference approximations, \cite[Lemmata 3.1, 3.2, Theorem 3.2]{Lady1985}, we get a strongly convergent subsequence to limit functions $u,v\in H^1(D_T)\cap \mathrm{Lip}([0,T];L^2(D))$. The limit functions satisfy the entropy inequality
\begin{equation}\label{eq:w11entropy}
\eta(u-k,v-\ell,a)_t+q(u-k,v-\ell)_x\le 0,\ \text{ in the sense of
	distributions,} 
\end{equation} 
where
\begin{equation}\label{eq:l2entropy}
\eta(u,v,a):=\frac{u^2}{2}+\frac{v^2}{2 a},\quad q(u,v):=-uv,
\end{equation}
which follows from \eqref{eq:wl2} in the limit $\Dx\rightarrow 0$.
They are therefore unique among solutions satisfying the entropy inequality (thanks to the linearity of the equation, we can insert another solution $(\widetilde{u}, \widetilde{v})$ for $(k,\ell)$).
\subsection{Convergence rate for the one dimensional wave
  equation.}\label{subsubsec:1drate}
In the last section, we showed that the numerical scheme
\eqref{eq:numwave1} converges to the weak solution of the 1-D wave
equation. However, the key question is the rate at which the
approximate solutions converge to the exact solution as the mesh is
refined, i.e., $\Dx \rightarrow 0$. The answer to this question is
provided in the following theorem,
  
 \begin{theorem}\label{lem:wave1d1}
   Let $a\in C^{0,\alpha}(\overline{\xdom})$ satisfy $\infty>\overline{a}\ge a(x)\ge
   \underline{a}>0$ for all $x\in D$. Denote by $(u,v)$ the solution
   of \eqref{eq:1dwavew} and $(u_\Dx, v_\Dx)$ the numerical
   approximation computed by the scheme \eqref{eq:numwave1} and
   defined in \eqref{eq:discfcn2}. Assume that the initial data
   $u_0,v_0\in L^2(D)$ and that $u,v, u_{\Dx}, v_{\Dx}$ have moduli of continuity
   \begin{equation}
   \label{eq:mocuv}
   \begin{split}
    \nu^2_x(u(t,\cdot),\sigma)\le C \, \sigma^{2\gamma},\quad
    &  \nu^2_x(v(t,\cdot),\sigma)\le C \, \sigma^{2\gamma},\\
     \nu^2_x(u_{\Dx}(t,\cdot),\sigma)\le C \, \sigma^{2\gamma},\quad
     & \nu^2_x(v_{\Dx}(t,\cdot),\sigma)\le C \, \sigma^{2\gamma} . 
     \end{split}  
   \end{equation}
   Then the approximation $(u_\Dx(t,\cdot),v_\Dx(t,\cdot))$ converges
   to the solution $(u(t,\cdot),v(t,\cdot))$, $0<t<T$, and we have the
   estimate on the rate
   \begin{multline}\label{eq:ratew1}
     \norm{(u-u_\Dx)(t,\cdot)}_{L^2(D)}+\norm{(v-v_\Dx)(t,\cdot)/a}_{L^2(D)}
     \\
     \le
     C\left(\norm{u_0-u_\Dx(0,\cdot)}_{L^2(D)}+\norm{(v_0-v_\Dx(0,\cdot))/a}_{L^2(D)}+
       \Dx^{(\alpha\gamma)/(2(\alpha\gamma+1-\gamma))}\right),
   \end{multline}
   where $C$ is a constant depending on $c$ and $T$ but not on $\Dx$.
 \end{theorem}
 \begin{remark}
 	If the initial data $u_0, v_0$ have moduli of continuity
 	\begin{equation*}
 	\nu_x^2(u_0,\sigma)\leq C \sigma^2,\quad \nu_x^2(v_0,\sigma)\leq C \sigma^2,
 	\end{equation*}
 	it follows from Lemma \ref{lem:prowave1} that $u,v$ have moduli of continuity in space and time with $\gamma=1$.
 \end{remark}
 \begin{proof}
   We let $\phi\in C^2_0((0,T)\times D)$ and define
   \begin{equation}\label{eq:lambdafcnwave}
     \Lambda_T(u,v,k,\ell,\phi):=\int_{D_T}
     \left(\frac{(u-k)^2}{2}+\frac{(v-\ell)^2}{2
         a}\right)\phi_t- (u-k)(v-\ell)\phi_x \,dxdt 
   \end{equation}
   The above definition is similar to the one used in Section \ref{sec:l2acvection} an adaptation of the Kru\v{z}kov doubling
   of variables technique \cite{HHNHRise07} in our current $L^2$
   setting.
We will use special test functions in $\Lambda_T$: We recall the definition of the mollifier $\omega_\epsilon$ in \eqref{eq:omegadelta} and define for some $0<\nu<\tau<T$,
   \begin{equation*}
     \psi^{\mu}(t):=H_\mu(t-\nu)-H_\mu(t-\tau),\quad
     H_\mu(t)=\int_{-\infty}^t \omega_\mu(\xi)\, d\xi. 
   \end{equation*}
   We define  $\Omega:\Dom^2\rightarrow \R$ by
   \begin{equation}\label{eq:testfcn1}
     \Omega(t,s,x,y)=\psi^\mu(t)\omega_{\epst}(t-s)\omega_{\epsx}(x-y).
   \end{equation}
   We choose $\nu$ and $\tau$ such that
   $0<\epst<\min\{\nu,T-\tau\}$ and
   $0<\mu<\min\{\nu-\epst,T-\tau-\epst\}$. We note that
   \begin{equation*}
     \Omega_t+\Omega_s=\psi^\mu_t \omega_{\epsx} \omega_{\epst},\quad \Omega_x+\Omega_y=0.
   \end{equation*} 
   We assume without loss of generality $\Dx\le
   \min\{\epsx,\epst,\nu\}$.  By the entropy inequality
   \eqref{eq:w11entropy}, we have for the solution $(u,v)$ of
   \eqref{eq:1dwavew} that
   $\Lambda_T(u,v,u_\Dx(s,y),v_\Dx(s,y),\phi)\ge 0$ for all $(s,y)\in
   D_T$ and test functions $\phi\in C^2_0((0,T)\times D)$. By
   \eqref{eq:entrw11}, we have on the other hand that
   \begin{multline}\label{eq:lala2}
     \int_{D_T}
     \left(\frac{(u_\Dx-u(t,x))^2}{2}+\frac{(v_{\Dx}-v(t,x))^2}{2
         a}\right)D^-_s\phi\,- (u_\Dx-u(t,x))(v_\Dx-v(t,x))\Dyc\phi
     \,dyds
     \\
     \ge
     \int_{D_T}(v_{\Dx}-v(t,x))^2\left(\frac{1}{2a}-\frac{1}{2a_\Dx}\right)
     D^-_s\phi\,
     dyds\\
     -\frac{\Dx^2}{2}(\theta-1)\int_{D_T}
     \left(D^+_y (u_\Dx-u) D^+_y (v_\Dx-v)\right)\,D_y^+\phi\, dyds\\
     +\frac{\Dx}{4}\int_{D_T}(D^+_y (v_\Dx-v(t,x))^2+D^+_y
     (u_\Dx-u(t,x))^2)D^+_y \phi\, dyds
   \end{multline}
   where $D^-_s\phi$ and $D^+_y \phi$ are defined by
   \begin{equation}\label{eq:diffquot}
     D^\pm_s\phi(s,y)=\mp\frac{\phi(s,y)-\phi(s\pm\Dt,y)}{\Dt},\quad D^\pm_y\phi(s,y)=\mp\frac{\phi(s,y)-\phi(s,y\pm \Dx)}{\Dx}.
   \end{equation}
   Adding $\Lambda_T(u,v,u_\Dx(s,y),v_\Dx(s,y),\phi)\ge 0$ and
   \eqref{eq:lala2}, choosing $\Omega$ as a test function and
   integrating over $D_T$, we obtain
   \begin{multline}\label{eq:ilbeltz}
     \underbrace{\int_{D_T^2}
       \left(\frac{(u_\Dx-u)^2}{2}+\frac{(v_{\Dx}-v)^2}{2
           a}\right)\left(\Omega_t+D^-_s\Omega\right)\,\duz}_A \\
     - \underbrace{\int_{D_T^2} (u_\Dx-u)(v_\Dx-v)\left(\Omega_x+\Dyc
         \Omega\right) \,\duz}_B
     \\
     \ge
     \underbrace{\int_{D_T^2}(v_{\Dx}-v)^2\left(\frac{1}{2a(x)}-\frac{1}{2a_\Dx(y)}\right)
       D^-_s\Omega\,\duz}_D
     \\
     +\underbrace{\frac{\Dx^2}{2}(\theta-1)\int_{D_T^2}\Dym\left[ \Dyp
         \left(u_\Dx-u\right)\Dyp\left( v_\Dx-v\right)\right] \Omega
       \, \duz}_E
     \\
     -\underbrace{\frac{\Dx}{4}\int_{D_T^2}((v_\Dx-v(t,x))^2+(u_\Dx-u(t,x))^2)D^-_y
       D^+_y \Omega\, \duz}_F
   \end{multline}
   We rewrite the term $A$ as
   \begin{align*}
     A&=\int_{D_T^2} \eta(u-u_\Dx,v-v_\Dx,a)(\Omega_t+D^-_s\Omega)\duz\\
     &=\underbrace{\int_{D_T^2} \eta(u-u_\Dx,v-v_\Dx,a)\psi^\mu_t
       \omega_{\epsx}
       \omega_{\epst}\,\duz}_{A_1}\\
     &\qquad+ \underbrace{\int_{D_T^2}
       \eta(u-u_\Dx,v-v_\Dx,a)\psi^\mu\,
       \omega_{\epsx}\left(\partial_t\omega_{\epst}+D^-_s\omega_{\epst}
       \right)\,\duz}_{A_2}
   \end{align*}
   The term $A_1$ can be written as
   \begin{equation*}
     A_1=\int_{D_T^2} \eta(u-u_\Dx,v-v_\Dx,a)\omega_\mu(t-\nu)
     \omega_{\epsx} \omega_{\epst}\,\duz-\int_{D_T^2}
     \eta(u-u_\Dx,v-v_\Dx,a)\omega_\mu(t-\tau) \omega_{\epsx}
     \omega_{\epst}\,\duz.
   \end{equation*}
   Introducing $\lambda$ as
   \begin{equation}
     \label{eq:lambdadef}
     \begin{aligned}
       \lambda(t)=\int_0^T\int_{D^2} &\eta\left(u_\Dx(s,y)-u(t,x),
         v_\Dx(s,y)-v(x,t),a(x)\right) \\
       &\times \omega_{\epsx}(x-y)\omega_{\epst}(t-s)\,dydxds,
     \end{aligned}
   \end{equation}
   we have that
   \begin{equation*}
     A_1 = \int_0^T \lambda(t)\omega_\mu(t-\nu)\,dt - \int_0^T
     \lambda(t)\omega_\mu(t-\tau)\,dt, 
   \end{equation*}
   so that \eqref{eq:ilbeltz} implies
   \begin{equation}\label{eq:ilbeltz2}
     \int_0^T \lambda(t)\omega_\mu(t-\nu)\,dt 
     +\abs{A_2}+\abs{B}+\abs{D}+\abs{E}+\abs{F}
     \ge\int_0^T
     \lambda(t)\omega_\mu(t-\tau)\,dt.
   \end{equation}
   Our task is now to overestimate $\abs{A_2}$, $\abs{B}$, $\abs{D}$,
   $\abs{E}$ and $\abs{F}$.

   To estimate the term $A_2$, we recall \eqref{eq:omegas} and observe that, (cf. \eqref{eq:null0})
   \begin{equation*}
     \frac{1}{\Dt} \int_0^T\int_0^\Dt
     \eta(u(t,x)-u_\Dx(t,y),v(t,x)-v_\Dx(t,y),a)
     (\xi-\Dt)\partial_{ss}\omega_{\epst}(t-s+\xi)\,d\xi ds= 0,
   \end{equation*}
   since all the terms in the integrand except
   $\partial_{ss}\omega_{\epst}(t-s+\xi)$ are independent of
   $s$. Therefore, subtracting this term from $A_2$, we obtain,
   \begin{multline*}
     A_2= \underbrace{\frac{1}{2\Dt}\int_{D_T^2}\int_0^\Dt
       (u_{\Dx}(t,y)-u_{\Dx})(2 u-u_{\Dx}- u_{\Dx}(t,y))\psi^\mu\,
       \omega_{\epsx}\,(\xi-\Dt)\partial_{ss}\omega_{\epst}(t-s+\xi)\,d\xi\duz}_{A_{2,1}}\\
     +\underbrace{\frac{1}{2\Dt}\int_{D_T^2}\int_0^\Dt
       \frac{1}{a}(v_{\Dx}(t,y)-v_{\Dx})(2v-v_{\Dx}-
       v_{\Dx}(t,y))\psi^\mu\,
       \omega_{\epsx}\,(\xi-\Dt)\partial_{ss}\omega_{\epst}(t-s+\xi)\,d\xi\duz}_{A_{2,2}}
   \end{multline*}
   We will outline the estimates for the term $A_{2,1}$, the term $A_{2,2}$
   is estimated in a similarly. By the triangle and H\"{o}lder's
   inequality
   \begin{align}\label{eq:a2w1}
     |A_{2,1}|\le & \frac{1}{2\Dt}\int_{D_T^2}\int_0^\Dt |u_{\Dx}(t,y)-u_{\Dx}(s,y)|\bigl( |u(t,x)-u_{\Dx}(s,y)|+|u(t,x)- u_{\Dx}(t,y)|\bigr)\\
     &\qquad\times\psi^\mu\,
     \omega_{\epsx}\,|\xi-\Dt|\, |\partial_{ss}\omega_{\epst}(t-s+\xi)|\,d\xi\duz\notag\\
     &\le \frac{1}{2\Dt}\int_0^{\Dt}\int_0^T\int_0^T \biggl( \int_{D^2} |u_{\Dx}(t,y)-u_{\Dx}(s,y)|^2\omega_{\epsx} \, dy\, dx\biggr)^{1/2}\notag\\
     &\qquad \times \biggl\{ \biggl( \int_{D^2} |u(t,x)-u_{\Dx}(s,y)|^2\omega_{\epsx} \, dy\, dx\biggr)^{1/2}\notag\\
     &\qquad\quad+\biggl( \int_{D^2} |u(t,x)-u_{\Dx}(t,y)|^2\omega_{\epsx} \, dy\, dx\biggr)^{1/2}\biggr\} \notag\\
     &\qquad\qquad \times\psi^\mu\,
     |\xi-\Dt|\, |\partial_{ss}\omega_{\epst}(t-s+\xi)|\, ds\, dtd\xi\,\notag\\
     &\le \frac{1}{2\Dt}\int_0^{\Dt}\int_0^T \sup_{{0\le s\le T}\atop{|t-s|<2\epst}}\biggl( \int_{D^2} |u_{\Dx}(t,y)-u_{\Dx}(s,y)|^2\omega_{\epsx} \, dy\, dx\biggr)^{1/2}\notag\\
     &\qquad \times \biggl\{ \sup_{{0\le s\le T}\atop{|t-s|<2 \epst}}\biggl( \int_{D^2} |u(t,x)-u_{\Dx}(s,y)|^2\omega_{\epsx} \, dy\, dx\biggr)^{1/2}\notag\\
     &\qquad\quad+\biggl( \int_{D^2} |u(t,x)-u_{\Dx}(t,y)|^2\omega_{\epsx} \, dy\, dx\biggr)^{1/2}\biggr\} \notag\\
     &\qquad\qquad \times\psi^\mu\,
     |\xi-\Dt|\, \int_0^T|\partial_{ss}\omega_{\epst}(t-s+\xi)|\, ds\, dt \,d\xi\notag\\
     &\le \frac{C}{\Dt\, \epst^{2-\gamma}}\int_0^{\Dt}\int_0^T
     \biggl\{ \sup_{{0\le s\le T}\atop{|t-s|<2 \epst}}\biggl( \int_{D^2} |u(t,x)-u_{\Dx}(s,y)|^2\omega_{\epsx} \, dy\, dx\biggr)^{1/2}\notag\\
     &\qquad\quad+\biggl( \int_{D^2}
     |u(t,x)-u_{\Dx}(t,y)|^2\omega_{\epsx} \, dy\,
     dx\biggr)^{1/2}\biggr\}\psi^\mu\,
     |\xi-\Dt|\, dt \,d\xi\notag\\
     &\le \frac{C\Dt}{\epst^{2-\gamma}}\int_0^{\Dt}\int_0^T
     \biggl\{ \sup_{{0\le s\le T}\atop{|t-s|<2 \epst}}\biggl( \int_{D^2} |u(t,x)-u_{\Dx}(s,y)|^2\omega_{\epsx} \, dy\, dx\biggr)^{1/2}\notag\\
     &\qquad\quad+\biggl( \int_{D^2}
     |u(t,x)-u_{\Dx}(t,y)|^2\omega_{\epsx} \, dy\,
     dx\biggr)^{1/2}\biggr\} \psi^\mu\, \, dt \notag
   \end{align}
   where we used the moduli of continuity for $u_\Dx$, %and $v_\Dx$,
   viz.~\eqref{eq:mocuv}, in the penultimate inequality and that
   $\Dt\le \epst$. Furthermore, in a similar way as we did for the advection equation in \eqref{eq:kappa1}, one can show that
    \begin{align}\label{eq:eis}
    \int_0^T&\sup_{{0\le s\le
    		T}\atop{|t-s|<\epst}}\left(\int_{D^2}\abs{u_\Dx(s,y)-u(t,x)}^2
    \omega_{\epsx}\, dydx\right)^{1/2}\psi^\mu\,dt
    \\
    &\le C T \epst^\gamma +\int_0^T
    \left(\int_0^T\int_{D^2}\abs{u_\Dx(s,y)-u(t,x)}^2
    \omega_{\epsx}\omega_{\epst}\,
    dydxds\right)^{1/2}\!\!\psi^\mu\,dt\notag
    \end{align}
   using the triangle inequality and similarly
   \begin{multline}\label{eq:zwei}
     \int_0^T\sup_{{0\le s\le T}\atop{|t-s|<\epst}}
     \left(\int_{D^2}\frac{1}{a}\abs{v_\Dx(s,y)-v(t,x)}^2\omega_{\epsx}\,
       dydx\right)^{1/2}\!\!\psi^\mu\, dt
     \\
     \le \frac{C T \epst^\gamma}{\underline{a}} +\int_0^T
     \left(\int_0^T\int_{D^2}\frac{1}{a}\abs{v_\Dx(s,y)-v(t,x)}^2
       \omega_{\epsx}\omega_{\epst}\,
       dydxds\right)^{1/2}\!\!\psi^\mu\,dt.
   \end{multline}
   Using $\lambda$, cf.~\eqref{eq:lambdadef}, \eqref{eq:a2w1} can be
   bounded as
   \begin{equation*}%\label{eq:a2w1nomal}
     \abs{A_{2,2}}\le C \Dt\,\epst^{2\gamma-2}+\frac{C
       \Dt}{\epst^{2-\gamma}}\int_0^T \sqrt{\lambda(t)}\,\psi^\mu\,
     dt.
   \end{equation*}
   and so, using a similar argument for the term $A_{2,1}$
   \begin{equation}\label{eq:a2w1nomal}
     \abs{A_{2}}\le C \Dt\,\epst^{2\gamma-2}+\frac{C
       \Dt}{\epst^{2-\gamma}}\int_0^T \sqrt{\lambda(t)}\,\psi^\mu\,
     dt.
   \end{equation}
   In order to bound the term $B$, we use
   \begin{align*}
     \Omega_x + \Dyc \Omega &= \frac{-1}{4\Dx} \int_0^\Dx (\xi-\Dx)^2
     \left[\partial_{yyy}\Omega(t,s,x,y-\xi)+\partial_{yyy}
       \Omega(t,s,x,y+\xi)\right]\,d\xi\\
     &= \frac{1}{4\Dx} \int_0^\Dx (\xi-\Dx)^2
     \left[\partial_{xxx}\Omega(t,s,x,y-\xi)+\partial_{xxx}
       \Omega(t,s,x,y+\xi)\right]\,d\xi.
   \end{align*}
   and that
   \begin{multline*}
     \frac{1}{4 \Dx}\int_0^\Dx\int_{D_T^2} (\xi-\Dx)^2
     \left(u_\Dx-u(t,y)\right)\left(v_\Dx-v(t,y)\right)\\
     \times\left[\partial_{xxx}\omega_{\epsx}(x-y+\xi)+\partial_{xxx}
       \omega_{\epsx}(x-y-\xi)\right]\omega_{\epst}\psi^\mu\,d\xi
     \duz =0,
   \end{multline*}
   since all the terms in the integrand, except
   $\left[\partial_{xxx}\omega_{\epsx}(x-y+\xi)+\partial_{xxx}
     \omega_{\epsx}(x-y-\xi)\right]$, are independent of $x$. We
   subtract this term from $B$ and add and subtract the term
   \begin{multline*}
     \frac{1}{4\Dx}\int_0^\Dx\int_{D_T^2} (\xi-\Dx)^2
     \left(u_\Dx-u(t,y)\right)\left(v_\Dx-v(t,x)\right)\\
     \times\left[\partial_{xxx}\omega_{\epsx}(x-y+\xi)+\partial_{xxx}
       \omega_{\epsx}(x-y-\xi)\right]\omega_{\epst}\psi^\mu\,d\xi
     \duz,
   \end{multline*}
   so that
   \begin{align*}
     B&= \frac{1}{4\Dx}\int_0^\Dx\int_{D_T^2} (\xi-\Dx)^2
     \left(u(t,y)-u(t,x)\right)\left(v_\Dx- v(t,x)\right)\\
     &\hphantom{=\frac{1}{4\Dx}\int_0^\Dx\int_{D_T^2}}\times\left[\partial_{xxx}\omega_{\epsx}(x-y+\xi)+\partial_{xxx}
       \omega_{\epsx}(x-y-\xi)\right]\omega_{\epst}\psi^\mu\,d\xi \duz\\
     &\hphantom{=}+\frac{1}{4\Dx}\int_0^\Dx\int_{D_T^2} (\xi-\Dx)^2
     \left(u_\Dx-u(t,y)\right)\left(v(t,y)- v(t,x)\right)\\
     &\hphantom{=\frac{1}{4\Dx}\int_0^\Dx\int_{D_T^2}}\times\left[\partial_{xxx}\omega_{\epsx}(x-y+\xi)+\partial_{xxx}
       \omega_{\epsx}(x-y-\xi)\right]\omega_{\epst}\psi^\mu\,d\xi \duz\\
     &:=B_1+B_2.
   \end{align*}
   We start by bounding $B_1$,
   \begin{align*}
     |B_1|&\le \frac{1}{4\Dx}\int_0^\Dx\int_{D_T^2} (\xi-\Dx)^2
     |u(t,y)-u(t,x)|\,|v_\Dx- v(t,x)|\\
     &\hphantom{\le \frac{1}{4\Dx}\int_0^\Dx\int_{D_T^2}}\times
     |\partial_{xxx}\omega_{\epsx}(x-y+\xi)+\partial_{xxx}
     \omega_{\epsx}(x-y-\xi)|\,\omega_{\epst}\psi^\mu\,d\xi \duz\\
     &\le \frac{1}{4\Dx}\int_0^\Dx \int_{D_T} \biggl(\int_{D_T}|u(t,y)-u(t,x)|^2\omega_{\epst}\, dy\, ds\biggr)^{1/2}\\
     &\hphantom{\le \frac{1}{4\Dx} \int_0^\Dx \int_{D_T}}
     \times\biggl(\int_{D_T}|v_\Dx- v(t,x)|^2\omega_{\epst}\, dy\, ds\biggr)^{1/2}(\xi-\Dx)^2  \\
     &\hphantom{\le \frac{1}{4\Dx}\int_0^\Dx \int_{D_T}}\times
     |\partial_{xxx}\omega_{\epsx}(x-y+\xi)+\partial_{xxx}
     \omega_{\epsx}(x-y-\xi)|\,\psi^\mu\,dx\, dt\,d\xi\\
     &\le \frac{1}{4\Dx}\int_0^\Dx \int_{0}^T \sup_{{x\ \mathrm{s.t.}}\atop{\abs{x-y}\le 3{\epsx}}}\biggl(\int_{D_T}|u(t,y)-u(t,x)|^2\omega_{\epst}\, dy\, ds\biggr)^{1/2}\\
     &\hphantom{\le \frac{1}{4\Dx} \int_0^\Dx \int_{D_T}}
     \times \sup_{{x\ \mathrm{s.t.}}\atop{\abs{x-y}\le 3{\epsx}}}\biggl(\int_{D_T}|v_\Dx- v(t,x)|^2\omega_{\epst}\, dy\, ds\biggr)^{1/2}(\xi-\Dx)^2  \\
     &\hphantom{\le \frac{1}{4\Dx}\int_0^\Dx \int_{D_T}}\times
     \int_D|\partial_{xxx}\omega_{\epsx}(x-y+\xi)+\partial_{xxx}
     \omega_{\epsx}(x-y-\xi)|\,dx\,\psi^\mu\, dt\,d\xi\\
     &\le \frac{C}{{\epsx}^{3-\gamma}\Dx}\int_0^\Dx \int_{0}^T \sup_{{x\ \mathrm{s.t.}}\atop{\abs{x-y}\le 3{\epsx}}}\biggl(\int_{D_T}|v_\Dx- v(t,x)|^2\omega_{\epst}\, dy\, ds\biggr)^{1/2}(\xi-\Dx)^2\psi^\mu\, dt\,d\xi\\
     &\le \frac{C \Dx^2}{{\epsx}^{3-\gamma}} \int_{0}^T \sup_{{x\
         \mathrm{s.t.}}\atop{\abs{x-y}\le
         3\epsx}}\biggl(\int_{D_T}|v_\Dx-
     v(t,x)|^2\omega_{\epst}\, dy\, ds\biggr)^{1/2}\psi^\mu\, dt
   \end{align*}
   where we have used that $\omega_{\epsx}$ is compactly supported in
   $[-\epsx,\epsx]$, and where $C$ is a constant depending on
   the $L^2$-norms and the moduli of continuity of the initial data
   and on $T$. Using that (c.f.~\eqref{eq:eis}, \eqref{eq:kappa2})
    \begin{align}\label{eq:drue}
    \int_0^T&\sup_{{x\ \mathrm{s.t.}}\atop{\abs{x-y}\le 3\epsx}}
    \biggl(\int_{D_T}\abs{u_\Dx(s,y)-u(t,x)}^2\omega_{\epst}\,
    dyds\biggr)^{1/2}\psi^\mu\,dt\\
    &\le C T \epsx^\gamma +\int_0^T
    \biggl(\int_0^T\int_{D^2}\abs{u_\Dx(s,y)-u(t,x)}^2
    \omega_{\epsx}\omega_{\epst}\,
    dydxds\biggr)^{1/2}\psi^\mu\,dt,\notag
    \end{align}
   and analogously,
   \begin{multline}\label{eq:zwei2}
     \int_0^T\sup_{{x\ \mathrm{s.t.}}\atop{\abs{x-y}\le
         3\epsx}}\biggl(\int_{D^T}\abs{v_\Dx(s,y)-v(t,x)}^2
     \omega_{\epst}\, dyds\biggr)^{1/2}\psi^\mu\,dt\\
     \le C T \epsx^\gamma +\int_0^T
     \biggl(\int_0^T\int_{D^2}\abs{v_\Dx(s,y)-v(t,x)}^2
     \omega_{\epsx}\omega_{\epst}\,
     dydxds\biggr)^{1/2}\psi^\mu\,dt,
   \end{multline}
   for $B_1$ we obtain the estimate
   \begin{equation}\label{eq:b1w1}
     |B_1|\le \frac{C\Dx^{2}}{\epsx^{3-2\gamma}}
     +\frac{C
       \Dx^{2}}{\epsx^{3-\gamma}}\int_0^T \sqrt{\lambda(t)}\,\psi^\mu\, dt.
   \end{equation}
   Similarly
   \begin{align*}
     |B_2|&\le \frac{1}{4\Dx}\int_0^\Dx\int_{D_T^2} (\xi-\Dx)^2
     |u_\Dx-u(t,y)|\,|v(t,y)- v(t,x)|\\
     &\hphantom{\le \frac{1}{4\Dx}\int_0^\Dx\int_{D_T^2}}\times
     |\partial_{xxx}\omega_{\epsx}(x-y+\xi)+\partial_{xxx}
     \omega_{\epsx}(x-y-\xi)|\,\omega_{\epst}\psi^\mu\,d\xi \duz\\
     &\le \frac{C \Dx^2}{\epsx^{3-\gamma}} \int_{0}^T
     \biggl(\int_{D_T}|v_\Dx(s,y)- v(t,y)|^2\omega_{\epst}\, dy\,
     ds\biggr)^{1/2}\psi^\mu\, dt
   \end{align*}
   Using \eqref{eq:zwei2}, we find, as for $B_1$,
   \begin{equation}\label{eq:b2w1}
     |B_2|\le \frac{C\Dx^{2}}{\epsx^{3-2\gamma}}
     +\frac{C
       \Dx^{2}}{\epsx^{3-\gamma}}\int_0^T \sqrt{\lambda(t)}\,\psi^\mu\, dt,
   \end{equation}
   and therefore
   \begin{equation}\label{eq:bw1}
     |B|\le \frac{C\Dx^{2}}{\epsx^{3-2\gamma}}
     +\frac{C
       \Dx^{2}}{\epsx^{3-\gamma}}\int_0^T \sqrt{\lambda(t)}\,\psi^\mu\, dt.
   \end{equation}
   We proceed to bound the term $D$. Observing that
   \begin{equation*}
     \int_{D_T^2}(v(t,x)-v_\Dx(t,y))^2\biggl(\frac{1}{2a(x)}-\frac{1}{2a_\Dx(y)}
     \biggr)D^-_s\Omega
     \,\duz=0, 
   \end{equation*}
   we can rewrite $D$ as
   \begin{equation*}
     D=\int_{D_T^2}\bigl(\left(v(t,x)-v_\Dx(t,y)\right)^2
     -\left(v(t,x)-v_\Dx(s,y)\right)^2\bigr)
     \biggl(\frac{1}{2a(x)}-\frac{1}{2a_\Dx(y)}\biggr)D^-_s\Omega
     \,\duz. 
   \end{equation*}
   Noting that,
   \begin{equation}\label{eq:vier}
     D_s^- \Omega(t,s,x,y)=\frac{1}{\Dt}\int^{\Dt}_0\Omega_s(t,s-\xi,x,y)\, d\xi,
   \end{equation}
   this becomes
   \begin{align*}
     D=\frac{1}{\Dt}\int_{D_T^2}\int_0^\Dt&(2v(t,x)-v_\Dx(t,y)-v_\Dx(s,y))\\
     &\quad
     \times(v_\Dx(t,y)-v_\Dx(s,y))\frac{a_\Dx(y)-a(x)}{2a(x)a_\Dx(y)}\Omega_s
     \,d\xi\,\duz,
   \end{align*}
   which can be bounded by
   \begin{align}\label{eq:cw1}
     |D|&\le \frac{1}{2\underline{a} \Dt}\sup_{|x-y|<\epsx}
     \abs{a(x)-a_\Dx(y)}
     \\
     &\quad\times\int_{D_T^2}\int_0^\Dt \frac{1}{c}
     \abs{2v(t,x)-v_\Dx(t,y)-v_\Dx(s,y)}\abs{v_\Dx(t,y)-v_\Dx(s,y)}
     \abs{\Omega_s} \,d\xi\,\duz
     \notag\\
     &\le \frac{C(\epsx+\Dx)^\alpha}{2\underline{a} \,\epst}
     \sup_{t\in (0,T)}\nu_t^2(v_\Dx(t,\cdot),\epst)^{1/2}
     \notag\\
     &\quad \times\int_0^T\sup_{{0\le s\le T}\atop{|t-s|<\epst}}
     \biggl(\int_{D^2}\frac{1}{a}\abs{v_\Dx(t,y)-v(s,x)}^2
     \omega_{\epsx}\, dydx \biggr)^{1/2}\,\psi^\mu\, dt
     \notag\\
     &\le \frac{C(\epsx+\Dx)^\alpha}{2\underline{a}
       \epst^{1-2\gamma}}+\frac{C(\epsx+\Dx)^\alpha}{2\underline{a}
       \epst^{1-\gamma}}\int_0^T\sqrt{\lambda(t)}\,\psi^\mu\,
     dt,\notag
   \end{align}
   where we have used \eqref{eq:zwei} for the last inequality.  For
   the term $E$, we note that it can be written
   \begin{equation*}
     E=\frac{\Dx^2}{2}(\theta-1)\int_0^T\int_{D_T}\Dym\left[ \Dyp
       u_\Dx\Dyp
       v_\Dx\right]\int_D \omega_{\epsx}(x-y)\, dx\,\omega_{\epst} \psi^\mu \, dy\, ds\, dt,    
   \end{equation*}
   so that
   \begin{align}\label{eq:Ebnd}
     E&=\frac{\Dx^2}{2}(\theta-1)\int_0^T\int_{D_T}\Dym\left[ \Dyp
       u_\Dx\Dyp v_\Dx\right]\,\omega_{\epst} \psi^\mu \, dy\, ds\, dt,\\
     &=\frac{\Dx^3}{2}(\theta-1)\int_0^T \int_0^T\sum_j\Dym\left[ \Dyp
       u_\Dx(s,x_j)\Dyp v_\Dx(s,x_j)\right]\,\omega_{\epst} \psi^\mu \, \, ds\, dt.\notag\\
     &=0\notag
   \end{align}
   In order to estimate the term $F$, we use that
   \begin{equation}\label{eq:dddiff}
     D^+_x D^-_x \phi(x)=\frac{1}{2 \Dx^2} \int_{-\Dx}^0\int_0^\Dx \phi''(x+\eta+\xi)\, d\xi\, d\eta,
   \end{equation}
   and that
   \begin{equation*}
     \frac{1}{8 \Dx} \int_{-\Dx}^0\int_0^\Dx\int_{D_T^2}((v_\Dx-v(t,y))^2+(u_\Dx-u(t,y))^2)
     \partial^2_x\omega_{\epsx}(x-y-\eta-\xi)\,\omega_{\epst}\psi^\mu \duz\, d\xi\, d\eta=0,
   \end{equation*}
   since all the terms in the integrand, but
   $\partial^2_x\omega_{\epsx}(x-y-\eta-\xi)$ are independent of
   $x$. We subtract this term from $F$ to find
   \begin{align*}
     F&=\underbrace{\frac{1}{8 \Dx} \int_{-\Dx}^0\int_0^\Dx\int_{D_T^2}(v-v(t,y))(v+v(t,y)-2v_\Dx)\partial^2_x\omega_{\epsx}(x-y-\eta-\xi)\,\omega_{\epst}\psi^\mu \duz\, d\xi\, d\eta}_{F_1}\\
     &+\underbrace{\frac{1}{8 \Dx}
       \int_{-\Dx}^0\int_0^\Dx\int_{D_T^2}(u-u(t,y))(u+u(t,y)-2
       u_\Dx)\partial^2_x\omega_{\epsx}(x-y-\eta-\xi)\,\omega_{\epst}\psi^\mu
       \duz\, d\xi\, d\eta}_{F_2}.
   \end{align*}
   The integrals $F_1$ and $F_2$ are estimated in the same way,
   therefore we outline only the estimate of $F_1$.
   \begin{align*}
     |F_1|&\le\frac{1}{8 \Dx}
     \int_{-\Dx}^0\int_0^\Dx\int_{D_T^2}|v-v(t,y)|\bigl(|v-v_\Dx|+|v(t,y)-v_\Dx|\bigr)| \partial^2_x\omega_{\epsx}|\,\omega_{\epst}\psi^\mu
     \duz\, d\xi\, d\eta\\
     &\le \frac{1}{8 \Dx} \int_{-\Dx}^0\int_0^\Dx\int_{0}^T\sup_{{x\
         \mathrm{s.t.}}\atop{\abs{x-y}\le 3\epsx}}\biggl(\int_{D_T}
     |v-v(t,y)|^2 \omega_{\epst}\, dy\, ds\biggr)^{1/2} \\
     &\hphantom{\le\frac{1}{8 \Dx}
       \int_{-\Dx}^0\int_0^\Dx\int_{0}^T}\times \biggl\{\sup_{{x\
         \mathrm{s.t.}}\atop{\abs{x-y}\le
         3\epsx}}\biggl(\int_{D_T}|v-v_\Dx|^2\omega_{\epst}\,
     dy\, ds\biggr)^{1/2}\\
     &\hphantom{\le\frac{1}{8 \Dx}
       \int_{-\Dx}^0\int_0^\Dx\int_{0}^T\biggl\{\sup_{{x\
           \mathrm{s.t.}}\atop{\abs{x-y}\le 3\epsx}}}+
     \biggl(\int_{D_T}|v(t,y)-v_\Dx|^2\omega_{\epst}\, dy\,
     ds\biggr)^{1/2}\biggr\}\\
     & \hphantom{\le\frac{1}{8 \Dx}
       \int_{-\Dx}^0\int_0^\Dx\int_{0}^T}\int_D
     |\partial^2_x\omega_{\epsx}|\,dx \,\psi^\mu \, dt d\xi\, d\eta\\
     &\le \frac{C \Dx}{\epsx^{2-\gamma}} \int_{0}^T \biggl\{
     \sup_{{x\ \mathrm{s.t.}}\atop{\abs{x-y}\le
         3\epsx}}\biggl(\int_{D_T}|v-v_\Dx|^2\omega_{\epst}\,
     dy\, ds\biggr)^{1/2}\\
     &\hphantom{\le \frac{C \Dx}{\epsx^{2-\gamma}} \int_{0}^T
       \sup_{{x\ \mathrm{s.t.}}\atop{\abs{x-y}\le
           3\epsx}}\biggl\{}+
     \biggl(\int_{D_T}|v(t,y)-v_\Dx|^2\omega_{\epst}\, dy\,
     ds\biggr)^{1/2}\biggr\} \psi^\mu \, dt.
   \end{align*}
   Using \eqref{eq:zwei2}, we find
   \begin{equation*}
     |F_1|\le \frac{C \Dx}{\epsx^{2-2\gamma}} +\frac{C
       \Dx}{\epsx^{2-\gamma}}  \int_{0}^T \sqrt{\lambda(t)}\psi^\mu
     \, dt, 
   \end{equation*}
   and therefore
   \begin{equation}\label{eq:ew1}
     \abs{F}\le \frac{C
       \Dx}{\epsx^{2-2\gamma}}
     +\frac{C\Dx}{\epsx^{2-\gamma}}\int_0^T\sqrt{\lambda(t)}\,\psi^\mu\,dt. 
   \end{equation}
   Referring to \eqref{eq:ilbeltz2}, we have established the following
   bounds
   \begin{align*}
     \abs{A_2}&\le C\left(\frac{\Dx}{\epst^{2-2\gamma}}
       +\frac{\Dx}{\epst^{2-\gamma}}\int_0^T \sqrt{\lambda(t)}\,
       \psi^\mu \,dt \right),\\
     \abs{B}&\le C\left(\frac{\Dx^{2}}{\epsx^{3-2\gamma}} +
       \frac{\Dx^{2}}{\epsx^{3-\gamma}} \int_0^T
       \sqrt{\lambda(t)}\,
       \psi^\mu \,dt\right),\\
     \abs{D}&\le C\left(\frac{\epsx^\alpha}{\epsx^{1-2\gamma}}
       + \frac{\epsx^\alpha}{\epst^{1-\gamma}}\int_0^T
       \sqrt{\lambda(t)}\,
       \psi^\mu \,dt\right),\\
     \abs{E}&= 0,\\
     \abs{F}&\le C\left(\frac{\Dx}{\epsx^{2-2\gamma}} +
       \frac{\Dx}{\epsx^{2-\gamma}}\int_0^T \sqrt{\lambda(t)}\,
       \psi^\mu \,dt\right),
   \end{align*}
   where we have used that $\Dt=C\Dx$ and $\Dx\le \epsx$.  Hence,
   \begin{multline*}
     \int_0^T\lambda(t)\omega_\mu(t-\tau)\,dt \le
     \int_0^T\lambda(t)\omega_\mu(t-\nu)\,dt +\underbrace{ C\left(
         \frac{\Dx}{\epst^{2-2\gamma}} +
         \frac{\Dx^{2}}{\epst^{3-2\gamma}} +
         \frac{\epsx^\alpha}{\epst^{1-2\gamma}}+\frac{\Dx}{\epsx^{2-2\gamma}} \right)}_{M_1}\\
     + \underbrace{
       C\left(\frac{\Dx}{\epst^{2-\gamma}}+\frac{\Dx^{2}}{\epsx^{3-\gamma}}
         +\frac{{\epsx}^\alpha}{\epst^{1-\gamma}}+
         \frac{\Dx}{\epsx^{2-\gamma}}\right)}_{M_2} \int_0^T
     \sqrt{\lambda(t)}\, \psi^\mu \,dt.
   \end{multline*}
   Sending $\mu$ to zero, we find
   \begin{equation*}
     \lambda(\tau) \le \lambda(\nu) + M_1 + M_2 \int_\nu^\tau \sqrt{\lambda(t)}\,dt.
   \end{equation*}
   With an application of a Gr\"{o}nwall type inequality, \cite[Chapter 1,
   Theorem 4]{gronwall1}, we obtain the estimate
   \begin{equation}\label{eq:dummichue}
     \lambda(\tau)\le \biggl(\sqrt{ \lambda(\nu)
       +M_1}+(\tau-\nu)M_2\biggr)^2\le
     2\left(\lambda(\nu) + M_1 + T^2 M_2^2 \right).
   \end{equation}
   By the triangle inequality, we have
   \begin{align}\label{eq:ex}
     \begin{split}
       \biggl|&\biggl(\int_D\int_{D_T}\abs{u_\Dx(s,x)-u(t,y)}^2
       \omega_{\epsx}\omega_{\epst}\, dxds dy\biggr)^{1/2}
       -\norm{u(t,\cdot)-u_\Dx(t,\cdot)}_{L^2(D)} \biggr|
       \\
       &\le \biggl(\int_D\int_{D_T}\abs{u_\Dx(s,x)-u_\Dx(t,y)}^2
       \omega_{\epsx}\omega_{\epst}\, dxdsdy\biggr)^{1/2}
       \\
       &\le \biggl(\int_{D^2}\abs{u_\Dx(t,x)-u_\Dx(t,y)}^2
       \omega_{\epsx}\, dxdy\biggr)^{1/2}
       +\biggl(\int_{D_T}\abs{u_\Dx(t,x)-u_\Dx(s,x)}^2
       \omega_{\epst}\, dsdx\biggr)^{1/2}
       \\
       &\le C \, (\epst^\gamma+\epsx^\gamma),
     \end{split}
   \end{align}
   and similarly
   \begin{multline}\label{eq:ex2}
     \biggl|\biggl(\int_D\int_{D_T}
     \frac{1}{a(x)}\abs{v_\Dx(t,x)-v(s,y)}^2
     \omega_{\epsx}\omega_{\epst}\, dxdsdy\biggr)^{1/2}
     -\norm{(v-v_\Dx)(t,\cdot)/a}_{L^2(D)} \biggr|
     \\
     \le C \, (\epst^\gamma+\epsx^\gamma).
   \end{multline}
   Moreover,
   \begin{equation}\label{eq:initte}
     \begin{aligned}
       &\norm{(u-u_\Dx)(\nu,\cdot)}_{L^2(D)}
       +\norm{(v-v_\Dx)(\nu,\cdot)/a}_{L^2(D)}
       \\
       &\quad\quad\le \norm{u_\Dx(\nu,\cdot)-u_\Dx(0,\cdot)}_{L^2(D)}
       +\norm{(v_\Dx(\nu,\cdot)-v_\Dx(0,\cdot))/a}_{L^2(D)}
       \\
       &\quad\quad\qquad+ \norm{u_0-u_\Dx(0,\cdot)}_{L^2(D)}
       +\norm{(v_0-v_\Dx(0,\cdot))/a}_{L^2(D)}
       \\
       &\quad\quad\qquad+\norm{u(\nu,\cdot)-u_0}_{L^2(D)}
       +\norm{(v(\nu,\cdot)-v_0)/a}_{L^2(D)}
       \\
       &\quad\quad\le C (\nu+\Dt)^\gamma
       +\norm{u_0-u_\Dx(0,\cdot)}_{L^2(D)}+\norm{(v_0-v_\Dx(0,\cdot))/a}_{L^2(D)}.
     \end{aligned}
   \end{equation}
   Write
   \begin{equation*}
     e(\tau)=\norm{(u-u_\Dx)(\tau,\cdot)}_{L^2(D)} +
     \norm{(v-v_\Dx)(\tau,\cdot)/a}_{L^2(D)}.
   \end{equation*}
   Thus, combining \eqref{eq:dummichue}, \eqref{eq:ex}, \eqref{eq:ex2}
   and \eqref{eq:initte}, the definition of $M_1$ and $M_2$ and some
   basic calculus inequalities, we obtain
   \begin{align}
     e^2(\tau)&\le C\biggl(e^2(0) + \epsx^{2\gamma} +
     \epst^{2\gamma} + \frac{\Dx}{\epst^{2-2\gamma}} +
     \frac{\epsx^\alpha}{\epst^{1-2\gamma}}+
     \frac{\Dx^2}{\epst^{4-2\gamma}}
     \label{eq:eest1}\\
     &\hphantom{\le C\Biggl(}\quad +
     \frac{\Dx^{4}}{\epsx^{6-2\gamma}} +
     \frac{\epsx^{2\alpha}}{\epst^{2(1-\gamma)}} +
     \frac{\Dx^{2}}{\epsx^{3-2\gamma}}+\frac{\Dx}{\epsx^{2-2\gamma}}+\frac{\Dx^2}{\epsx^{4-2\gamma}}
     \biggr).\notag
   \end{align}
   Hence, choosing $\epsx=\epst^{1/\alpha}$ and
   $\epsx=\Dx^{1/(2(\gamma\alpha+1-\gamma))}$,
   \begin{equation*}
     e(\tau)
     \le
     C\left(e(0)+\Dx^{(\alpha\gamma)/(2(\alpha\gamma+1-\gamma))}\right).
   \end{equation*}
 \end{proof}
\begin{remark}
	We note that for $\gamma=1$, this reduces to a rate of $\Dx^{1/2}$ independently of $\alpha$.
\end{remark}

 \subsection{Numerical examples}\label{ssec:numericalexampleswave}
 Next, we shall compare the above derived convergence rates to the
 ones obtained in practice. To this end, we implement the finite
 difference scheme \eqref{eq:numwave1} and test it on a set of
 numerical test cases. For all the test cases, we use the interval
 $D=[0,2]$ as the computational domain with periodic boundary
 conditions.
 We use again the sample of a log-normally distributed random field from Section \ref{ssec:exprates} as a material coefficient $a$ (c.f. Figure \ref{fig:coef1d}).
 We compute approximations at time $T=1$ and test the scheme in this
setting  with different choices of initial data. We only test the case that the initial data $v_0, u_0$ have a moduli of continuity $\gamma=1$ for which we could show in Lemma \ref{lem:prowave1} that the solutions have the same moduli of continuity. In this case, Theorem \ref{lem:wave1d1} predicts a rate of convergence of  $\nicefrac{1}{2}$. Specifically, we run experiments with initial data
\begin{equation}
\label{eq:initw1}
	v_{0,1}(x)=\sin(2\pi x), \quad u_{0,1}(x)=\cos(2\pi x),
\end{equation}
with
\begin{equation}
\label{eq:initw2}
v_{0,2}(x)=\begin{cases}
-1-2(x-\nicefrac{1}{2}),&\quad x\in[0,0.5)\\
-1+2(x-\nicefrac{1}{2})&\quad x\in[0.5,1),\\
-1-2(x-\nicefrac{3}{2}),&\quad x\in[1,1.5),\\
-1+2(x-\nicefrac{3}{2}),&\quad x\in[1.5,2),
\end{cases}
\quad u_{0,2}(x)=\begin{cases}
1+2(x-\nicefrac{1}{2}),&\quad x\in[0,0.5)\\
1-2(x-\nicefrac{1}{2})&\quad x\in[0.5,1),\\
1+2(x-\nicefrac{3}{2}),&\quad x\in[1,1.5),\\
1-2(x-\nicefrac{3}{2}),&\quad x\in[1.5,2),
\end{cases}
\end{equation}
(note that $v_{0,2}=-u_{0,2}$). As a third set of initial data we take $v_{0,3}=v_{0,2}$ and for $u_{0,3}$ we take the composition of $30$ random hat functions on $[0,2]$, i.e.
\begin{equation*}
u_{0,3}=\sum_{j=1}^{30} h_j(x),
\end{equation*}
where $h_j$ is given by
\begin{equation*}
h_j(x)=\begin{cases}
0, &\quad x\in[0,x_0]\cup (x_2,2]\\
q \frac{x-x_0}{x_1-x_0},&\quad x\in(x_0,x_1],\\
q\frac{x_2 - x}{x_2-x_1},&\quad x\in(x_1,x_2],
\end{cases}
\end{equation*}
where $q\sim \mathcal{U}(-1,1)$, $x_0\sim \mathcal{U}(0,1)$, $x_1\sim \mathcal{U}(x_0,2)$ and $x_2\sim \mathcal{U}(x_1,2)$ are samples of uniformly distributed random variables. The initial data $u_{0,3}$ and $v_{0,3}$ are pictured in Figure \ref{fig:initdataw} on the left, and the approximation of the linear wave equation by scheme \eqref{eq:numwave1} at time $T=1$ on the right (on a grid with $2^{14}$ points).

\begin{figure}[ht] % \caption{A gull}
	\centering
	\begin{tabular}{lr}
		\includegraphics[width=0.5\textwidth]{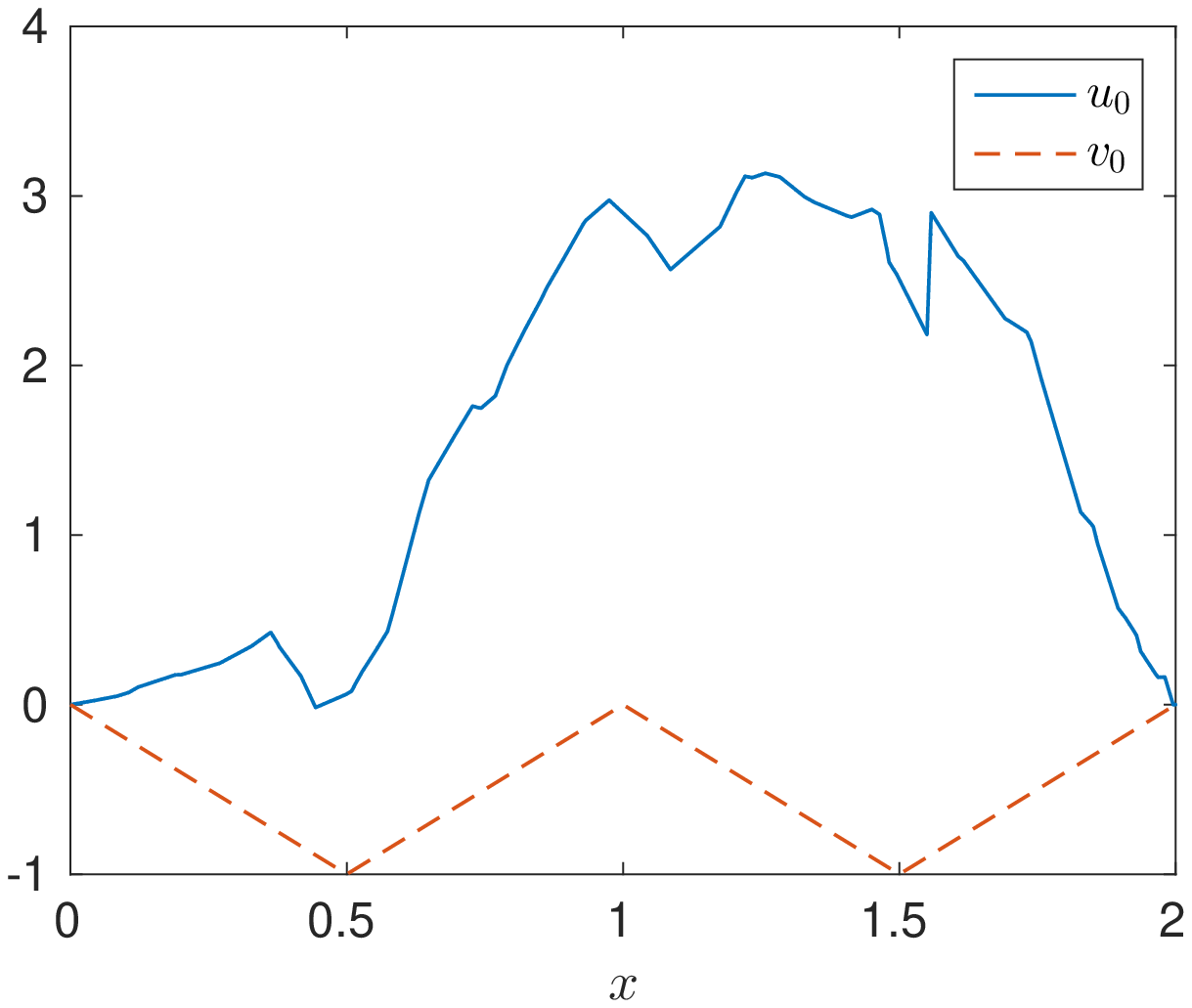}
		&\includegraphics[width=0.5\textwidth]{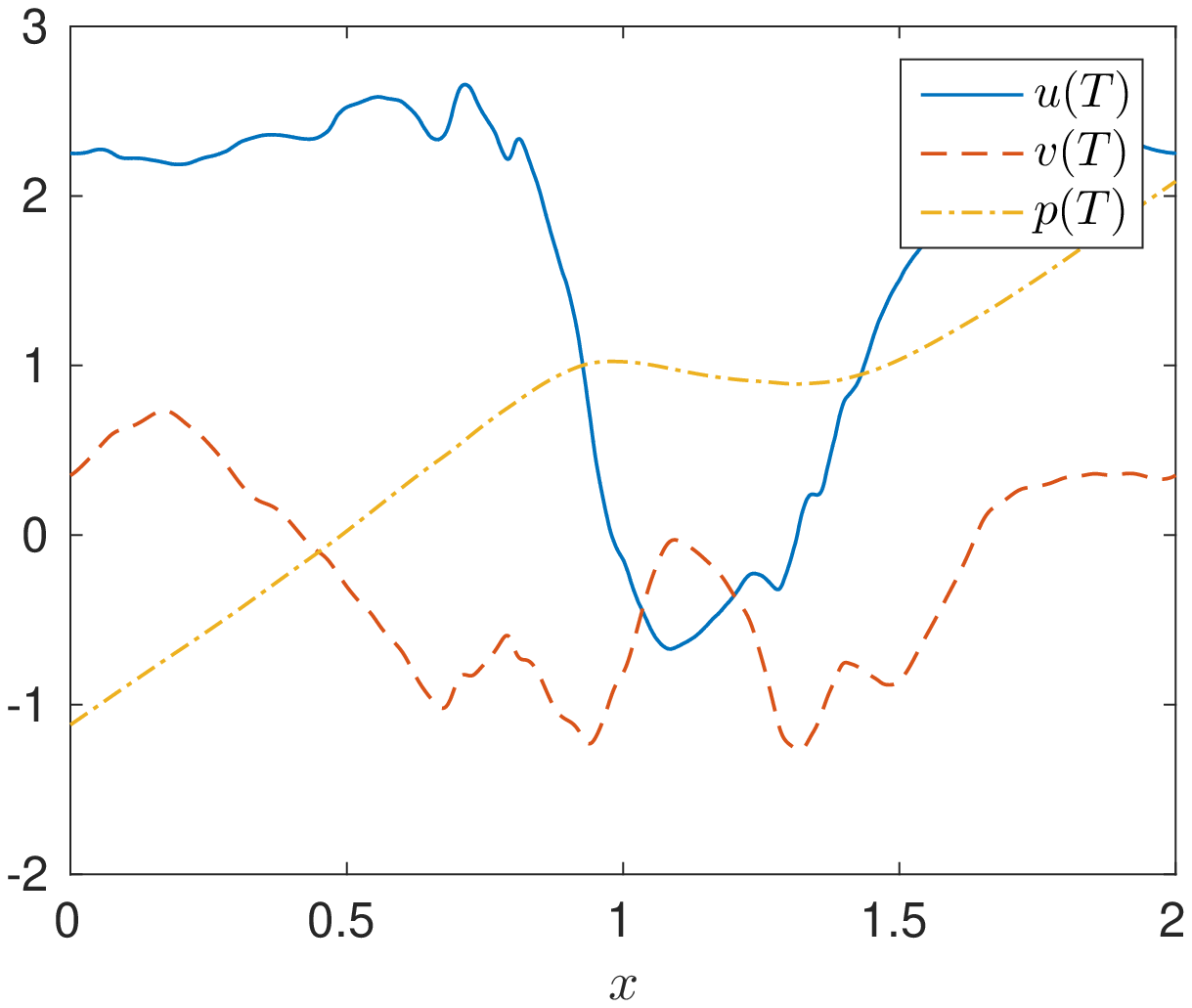}
	\end{tabular}
	\caption{Left: Initial data $u_{0,3}$ and $v_{0,3}$, right: solution at time $T=1$ for initial data $u_{0,3},v_{0,3}$.}
	\label{fig:initdataw}
\end{figure}

The above chosen initial data have moduli of continuity of $\gamma = 1$ in $L^2$ since they are contained in $H^1([0,2])$. We ran $N_{\text{exp}} = 6$ experiments for each set of initial data for mesh resolutions $\Dx = 2^{-5}, \dots, 2^{-10}$ (i.e. $N_x=2^6,\dots, 2^{11}$) and compute errors and rates as in \eqref{eq:relerr}, \eqref{eq:rateapprox} for $m=2$ against a reference solution computed on a grid with $2^{14}$ points. The obtained rates for the three sets of initial data are displayed in Table \ref{tab:wr}.

\begin{table}[h]
	\centering
	\begin{tabular}[h]{|c|c|c|c|}
		\hline
		  &$ r^2_u$ & $r^2_v$ & $r^2_p$ \\
		\hline
		$u_{0,1}, v_{0,1}$ & 0.8715 & 0.6968 & 0.9170\\
		$u_{0,2}, v_{0,2}$ & 0.7424 & 0.7542 & 0.9214 \\
		$u_{0,3}, v_{0,3}$ & 0.5992 & 0.5868 & 0.9503\\
		\hline
		%	\hline
		%\hspace{1mm}
	\end{tabular}
	\vspace{1mm}
	\caption{Experimental rates, $r^2_u$ is the rate for $u$, $r^2_v$ the rate for $v$ and $r^2_p$ the rate for $p$.}
	\label{tab:wr}
\end{table}

We observe that the rates are higher than the $\nicefrac{1}{2}$ which the theory predicts, but not by much, as the example with initial data $u_{0,3}, v_{0,3}$ shows. Moreover, we are testing self-convergence, so the actual convergence rate could be slower. We have also computed convergence rates for the approximation of $p$, which we computed by integrating the approximation of $v$ with a forward Euler scheme, i.e.
\begin{equation*}
p^{n+1}_j = p^n_j+\Dt \,v_j^n,
\end{equation*}
for which we have not proved any theoretical results. We observe that the rate for this variable is higher, close to $1$, which is probably due to the fact that $p$ has more regularity than $v$ and $u$, as it can be written as an integral of either of those.

\section{Conclusions}
\label{sec:conc}
Acoustic waves that propagate in a heterogeneous medium, for instance
an oil and gas reservoir, are modeled using the linear wave equation
\eqref{eq:waveeq} with a variable material coefficient $c$. Standard
finite difference and finite element approximations 
converge to the solution as the mesh is refined. A rate of convergence
for these approximations can be obtained based on the assumption that the
underlying solution is smooth enough. This requires enough smoothness
of the material coefficient (wave speed).

However in many practical situations of interest such as seismic wave
imaging and hydrocarbon exploration, the material coefficient is not
smooth, not even continuously differentiable. Moreover, the
material coefficient (rock permeability) is usually modeled by a
log-normal random field. Path-wise realizations of such fields are at
most H\"older continuous. Thus, the design of numerical schemes that
can approximate wave propagation in H\"older continuous media is a
necessary first step in the efficient solution of the underlying
uncertain PDE with a log-normal distributed material coefficient
\cite{sssid3}. We are not aware of rigorous numerical analysis results for discretizations of the wave equation with such rough coefficients apart from the works~\cite{Jovanovic1987,FDMbook,Jovanovic1992} which require the coefficient to be in $W^{s,2}(D)$ for some $s\in(1,3]$.

The current paper is therefore an attempt to design robust numerical
approximations for the one-dimensional transport and the wave equation with rough, i.e., only H\"older continuous coefficients. For low enough H\"{o}lder exponent, this regularity requirement is less than the one in~\cite{Jovanovic1987,FDMbook,Jovanovic1992}, and also our assumptions on the regularity of the solution are weaker. However, our results (so far) restrict to the one-dimensional case.

We propose {upwind} finite difference
approximations and show that these approximations converge as the mesh is
refined. Furthermore, we establish rigorous convergence rates of these approximations. The
obtained rates explicitly depend on the H\"older exponent of the
material coefficient as well as the modulus of continuity in $L^1$ or $L^2$ of
the initial data. The rates of convergence are obtained by a
novel adaptation of the Kru\v{z}kov doubling of variables
technique from scalar conservation laws to our $L^2$ linear system
setting. In particular, we prove that for coefficients which are H\"{o}lder continuous with exponent $\alpha$ and initial data that is H\"{o}lder continuous with exponent $\gamma$, the solution of the transport equation and its approximation have the same H\"{o}lder continuity and the approximations converge with rate at least $(\gamma\alpha)/(\gamma\alpha + 2-\gamma)$ in $L^1$ and $L^2$ (c.f. Theorem \ref{lem:ratetransport} and \ref{lem:ratetransportl2}). For the wave equation, we could show that if the initial data has a modulus of continuity of $\gamma=1$ in $L^2$, then the solution will inherit it. In this case, the finite difference approximations converge at rate at least $\nicefrac{1}{2}$. The numerical experiments demonstrate the near sharpness of this rate. 
We also show rates of convergence under the assumption that the numerical approximations have lower moduli of continuity, however, in this case, we cannot prove that the numerical approximations actually inherit those.

We conclude with a brief discussion on possible limitations and future
extensions of our methods:
\begin{itemize}
\item We consider finite difference discretizations in the current
  paper. The formal order of accuracy of our three-point finite
  difference schemes is one. One can argue that analogous to linear
  hyperbolic systems with smooth coefficients, one can obtain higher
  rates of convergence by designing schemes with a larger stencil (a
  higher formal order of accuracy). We find that prospect unlikely to
  hold in practice on account of the lack of smoothness of the
  coefficient. Furthermore, the irregularities of the coefficient are
  not localized. Hence, one cannot expect any localization of
  singularities in the solution and its derivatives. This is in marked
  contrast to nonlinear systems of conservation laws where
  discontinuities such as shocks and contact discontinuities separate
  smooth parts of the flow. Thus, high-resolution finite difference
  schemes perform better than low order schemes for conservation
  laws. Such a situation does not hold for wave propagation in a rough
  medium. We expect that the low-order schemes presented here are not
  only simple but also optimal in this case.

\item We present the analysis only in one space dimension and for uniform grids. The extension of the finite difference scheme to the two and three dimensional wave equation is straightforward, however, it is not easy to show that the solution computed in this way has a modulus of continuity, which is fundamental to obtaining convergence rates using our technique. In fact, we do not know currently if the approximations do have a modulus of continuity. Obtaining more insight into the regularity of the approximations is one of the objectives of our current research efforts. We would furthermore like to extend the method and convergence analysis to unstructured grids.
  
\item The numerical experiments for the transport equation (Section \ref{ssec:exprates}) suggest that the rate from Theorem \ref{lem:ratetransport} may not be sharp. Consequently, we plan to experiment more in order to find out if the rate is sharp or not, and otherwise try to improve the estimate.

\item We restrict ourselves to acoustic wave propagation in rough
  media in this paper. However, elastic wave propagation also involve
  media with material properties that lead to rough, H\"older
  continuous coefficients. The extension of these methods to such
  problems will be considered in a forthcoming paper. Another possible
  direction of research would be to prove a rate of convergence for
  numerical methods that approximate electromagnetic wave propagation
  in heterogeneous media. Possible extensions to nonlinear wave
  equations, and discontinuous and time-dependent coefficients will also be considered.

\end{itemize}
\section{Acknowledgments}
%\thanks{}
	This project was supported in part by ERC STG. N 306279, SPARCCLE.
	The author thanks Jonas \v{S}ukys (ETH) for providing the sample coefficient and helpful discussions on the topic. Many thanks also to Nils Henrik Risebro (UiO) and Siddhartha Mishra (ETHZ) for discussions and advice in the course of preparing this manuscript.

\end{document}